\documentclass[11pt,reqno,tbtags,a4paper]{amsart}
\usepackage{amssymb}
\usepackage{mathabx}  
\usepackage{xpunctuate}
\usepackage{url}
\usepackage[square,numbers]{natbib}
\bibpunct[, ]{[}{]}{;}{n}{,}{,}
\usepackage{epic}

\title[Fringe trees of full binary trees]
{Fringe trees of Patricia tries, compressed binary search trees, 
and three other random full binary trees}

\date{2 May, 2024; revised 30 January, 2026}

\author{Svante Janson}
\thanks{Supported by the Knut and Alice Wallenberg Foundation
and
the Swedish Research Council
}
\address{Department of Mathematics, Uppsala University, PO Box 480,
SE-751~06 Uppsala, Sweden}
\email{svante.janson@math.uu.se}


\numberwithin{equation}{section}

\allowdisplaybreaks

\theoremstyle{plain}
\newtheorem{theorem}{Theorem}[section]
\newtheorem{lemma}[theorem]{Lemma}

\newtheorem{corollary}[theorem]{Corollary}
{Claim}

\theoremstyle{definition}

\newcommand\xqed[1]{%
    \leavevmode\unskip\penalty9999 \hbox{}\nobreak\hfill
    \quad\hbox{#1}}

\newtheorem{exampleqqq}[theorem]{Example}
\newenvironment{example}{\begin{exampleqqq}}
  {\xqed{$\triangle$}\end{exampleqqq}}

\newenvironment{examplet}[1]
  {\begin{example}{\ensuremath{\mathbf{#1}}\,\textbf:}}
  {\end{example}}

\newtheorem{remarkqqq}[theorem]{Remark}
\newenvironment{remark}{\begin{remarkqqq}}
  {\xqed{$\triangle$}\end{remarkqqq}}

\newtheorem{problem}[theorem]{Problem}

\theoremstyle{remark}

\newcounter{dummy}
\makeatletter
\newcommand\myitem[1][]{\item[#1]\refstepcounter{dummy}\def\@currentlabel{#1}}
\makeatother

\newenvironment{romenumerate}[1][-10pt]{
\addtolength{\leftmargini}{#1}\begin{enumerate}
 }{\end{enumerate}}

\newenvironment{rommenumerate}[1][-10pt]{
\addtolength{\leftmargini}{#1}\begin{enumerate}
 }{\end{enumerate}}

\newenvironment{alphenumerate}[1][-10pt]{
\addtolength{\leftmargini}{#1}\begin{enumerate}
 }{\end{enumerate}}

\newcounter{oldenumi}
{\setcounter{oldenumi}{\value{enumi}}
\begin{romenumerate} \setcounter{enumi}{\value{oldenumi}}}
{\end{romenumerate}}

\newcounter{thmenumerate}

\newcounter{xenumerate}   

\newcommand\pfitemx[1]{\par#1:}
\newcommand\pfitemref[1]{\pfitemx{\ref{#1}}}

\newcommand{\refT}[1]{Theorem~\ref{#1}}

\newcommand{\refC}[1]{Corollary~\ref{#1}}
\newcommand{\refCs}[1]{Corollaries~\ref{#1}}
\newcommand{\refL}[1]{Lemma~\ref{#1}}
\newcommand{\refLs}[1]{Lemmas~\ref{#1}}
\newcommand{\refR}[1]{Remark~\ref{#1}}

\newcommand{\refS}[1]{Section~\ref{#1}}
\newcommand{\refSs}[1]{Sections~\ref{#1}}
\newcommand{\refSS}[1]{Section~\ref{#1}}
\newcommand{\refSSS}[1]{Section~\ref{#1}}

\newcommand{\refF}[1]{Figure~\ref{#1}}

\begingroup
  \divide\count255 by 60
  \multiply\count255 by -60
  \advance\count255 by \time
  \ifnum \count255 < 10 \xdef\klockan{\the\count1.0\the\count255}
  \else\xdef\klockan{\the\count1.\the\count255}\fi
\endgroup

\newcommand{\sumko}{\sum_{k=0}^\infty}

\newcommand{\sumk}{\sum_{k=1}^\infty}
\newcommand{\suml}{\sum_{\ell=1}^\infty}

\newcommand\set[1]{\ensuremath{\{#1\}}}

\newcommand\xpar[1]{(#1)}
\newcommand\bigpar[1]{\bigl(#1\bigr)}
\newcommand\Bigpar[1]{\Bigl(#1\Bigr)}

\newcommand\bigsqpar[1]{\bigl[#1\bigr]}

\newcommand\Bigsqpar[1]{\Bigl[#1\Bigr]}

\newcommand\cpar[1]{\{#1\}}

\newcommand\abs[1]{\lvert#1\rvert}
\newcommand\bigabs[1]{\bigl\lvert#1\bigr\rvert}

\newcommand\lrabs[1]{\left\lvert#1\right\rvert}
\def\rompar(#1){\textup(#1\textup)}    

\def\xexp(#1){e^{#1}}

\newcommand\floor[1]{\lfloor#1\rfloor}

\newcommand\frax[1]{\{#1\}}

\newcommand\ntoo{\ensuremath{{n\to\infty}}}

\newcommand\mtoo{\ensuremath{{m\to\infty}}}

\newcommand\punkt{\xperiod}    
    
\newcommand\ie{i.e\punkt}
\newcommand\eg{e.g\punkt}

\newcommand\cf{cf\punkt}

\newcommand\ii{\mathrm{i}}

\newcommand{\tend}{\longrightarrow}
\newcommand\dto{\overset{\mathrm{d}}{\tend}}
\newcommand\pto{\overset{\mathrm{p}}{\tend}}

\newcommand\eqd{\overset{\mathrm{d}}{=}}

\newcommand\op{o_{\mathrm p}}

\newcommand\bbR{\mathbb R}

\newcommand\bbN{\mathbb N}

\newcommand\bbQ{\mathbb Q}
\newcommand\bbZ{\mathbb Z}

\newcounter{CC}
\newcounter{cc}

\renewcommand\Re{\operatorname{Re}}

\newcommand\E{\operatorname{\mathbb E}{}} 
\renewcommand\P{\operatorname{\mathbb P{}}}

\newcommand\Var{\operatorname{Var}}
\newcommand\Cov{\operatorname{Cov}}

\newcommand\Po{\operatorname{Po}}

\newcommand\Be{\operatorname{Be}}

\newcommand\ga{\alpha}
\newcommand\gb{\beta}
\newcommand\gd{\delta}

\newcommand\gf{\varphi}
\newcommand\GF{\Phi}
\newcommand\gam{\gamma}

\newcommand\gG{\Gamma}

\newcommand\kk{\kappa}
\newcommand\gl{\lambda}

\newcommand\gO{\Omega}
\newcommand\gs{\sigma}

\newcommand\gss{\sigma^2}

\newcommand\gU{\Upsilon}
\newcommand\eps{\varepsilon}
\renewcommand\phi{\xxx}  

\newcommand\cA{\mathcal A}
\newcommand\cB{\mathcal B}

\newcommand\cD{\mathcal D}

\newcommand\cL{{\mathcal L}}

\newcommand\cU{{\mathcal U}}
\newcommand\cUx{\cU^*}

\newcommand\cX{{\mathcal X}}

\newcommand\indic[1]{\boldsymbol1\cpar{#1}}

\newcommand\qw{^{-1}}

\newcommand\intoi{\int_0^1}

\newcommand\intox{\int_0^x}
\newcommand\intoo{\int_0^\infty}

\newcommand\oi{\ensuremath{[0,1]}}

\newcommand\setoi{\set{0,1}}

\newcommand\dd{\,\mathrm{d}}

\newcommand\rhs{right-hand side}

\newcommand\xT{\widetilde{\TR}}
\newcommand\Tqq[1]{\xT_{#1}}
\newcommand\Tgl{\Tqq{\gl}}

\newcommand\fA{{\mathfrak S}}
\newcommand\fT{{\mathfrak T}}
\newcommand\fTo{{\fT_0}}
\newcommand\hfT{{\widehat{\mathfrak T}}}
\newcommand\fTsett[1]{\mathfrak{T}^{#1}}
\newcommand\fTset[1]{\mathfrak{T}^{\set{#1}}}
\newcommand\fTp{\mathfrak{T}'}
\newcommand\fTPath{{\fT_P}}
\newcommand\cht{\widecheck{\fT}_t}
\newcommand\chtt{\widecheck{\fT}_t^+}
\newcommand\chtx[1]{\widecheck{\fT}_{#1}^+}

\newcommand\bT{{\overline{T}}}
\newcommand\hT{{\widehat{T}}}
\newcommand\tT{T}
\newcommand\Tx{T^*}
\newcommand\Txoo{\Tx_\infty}
\newcommand\TI{\bullet}
\newcommand\BST{\cB}
\newcommand\BSTx{\BST^*}

\newcommand\BSTxoo{\BSTx_\infty}

\newcommand\BSTY{\cB'}
\newcommand\eBST{\overline{\BST}}
\newcommand\eBSTx{\overline{\BST}^*}
\newcommand\eBSTxx{\overline{\BST}^{**}}
\newcommand\eBSTxoo{\eBSTx_\infty}
\newcommand\eBSTxxoo{\eBSTxx_\infty}
\newcommand\et{{\overline{t}}}
\newcommand\hBST{\widehat{\BST}}
\newcommand\hBSTx{\widehat{\BST}^*}
\newcommand\hBSTxx{\widehat{\BST}^{**}}
\newcommand\hBSTxoo{\hBSTx_\infty}
\newcommand\hBSTxxoo{\hBSTxx_\infty}
\newcommand\TR{\gU}
\newcommand\hTR{\widehat{\TR}}
\newcommand\Tn{\TR_n}

\newcommand\hTn{\hTR_n}

\newcommand\hTnx{\hTR_n^*}
\newcommand\hTnxx{\hTR_n^{**}}
\newcommand\CB{\cD}
\newcommand\CBx{\CB^*}
\newcommand\ld{d_{\mathsf{L}}}
\newcommand\rd{d_{\mathsf{R}}}
\newcommand\LPL{\mathsf{LPL}}
\newcommand\RPL{\mathsf{RPL}}
\newcommand\sss[1]{^{(#1)}}
\newcommand\esize[1]{|#1|_{\mathsf{e}}}
\newcommand\isize[1]{|#1|_{\mathsf{i}}}
\newcommand\dpp{d_p}

\newcommand\sxx[1]{\mathsf{#1}}
\newcommand\sE{\sxx{E}}
\newcommand\sV{\sxx{V}}
\newcommand\sC{\sxx{C}}
\newcommand\sX{\sxx{X}}
\newcommand\fE{f_{\sE}}
\newcommand\fV{f_{\sV}}
\newcommand\fC{f_{\sC}}
\newcommand\fX{f_{\sX}}

\newcommand\Mx[1]{#1^*}

\newcommand\MfE{\Mx{\fE}}
\newcommand\MfV{\Mx{\fV}}
\newcommand\MfC{\Mx{\fC}}
\newcommand\MfX{\Mx{\fX}}

\newcommand\psiE{\psi_{\sE}}
\newcommand\psiV{\psi_{\sV}}
\newcommand\psiC{\psi_{\sC}}
\newcommand\psiX{\psi_{\sX}}

\newcommand\sL{{\mathsf{L}}}
\newcommand\sR{{\mathsf{R}}}

\newcommand{\sumkoooo}{\sum_{k=-\infty}^\infty}

\newcommand\cAx{\cA^*}

\newcommand\sumaxx{{\sum_{\ga}}'}
\newcommand\tsumaxx{\sum_{\ga}'}
\newcommand\TXX{{T}^{**}}
\newcommand\TXXX[1]{{T}^{**(#1)}}
\newcommand\TXXoo{\TXX_\infty}
\newcommand\TXXXoo[1]{\TXXX{#1}_\infty}

\newcommand\CBXX{{\CB}^{**}}
\newcommand\CBXXX[1]{{\CB}^{**(#1)}}
\newcommand\CBXXoo{\CBXX_\infty}
\newcommand\CBXXXoo[1]{\CBXXX{#1}_\infty}
\newcommand\qsin{q}
\newcommand\chio{\frac{2\pi}{\log 2}}
\newcommand\ppi{p}
\newcommand\psio{\psi_0}
\newcommand\hpsio{\widehat{\psi}_0}
\newcommand\Law{\mathfrak L}
\newcommand\gammu{\gamma^2_U}

\newcommand\DTCS{\mathrm{DTCS}}
\newcommand\internal{^\circ}
\newcommand\hgb{\hat\gb}
\newcommand\hgam{\hat\gamma}
\newcommand\hgamm{\hat\gamma^2}
\newcommand\egb{\bar\gb}
\newcommand\egam{\bar\gam}
\newcommand\egamm{\bar\gam^2}
\newcommand\tint{{t\internal}}
\newcommand\mux[1]{\mu(#1)}
\newcommand\suminu{\sum_{i=1}^\nu}

\begin{document}

\begin{abstract} 
We study the distribution of fringe trees in 
Patricia tries (extending earlier results by Ischebeck (2025))
and compressed binary search trees; both cases are random binary trees that
have been 
compressed by deleting nodes of outdegree 1 so that they are random full
binary trees. The main results are central limit theorems for the number of
fringe trees of a given type, which imply quenched and annealed limit
results for the fringe tree distribution; for Patricia tries, this is
complicated by periodic oscillations in the usual manner.
We also consider extended fringe trees.
The results are derived from earlier results
for uncompressed tries and binary search trees.
In the case of compressed binary search trees, it seems difficult to give a
closed formula for the asymptotic fringe tree distribution, but we provide a
recursion and give examples.

For comparison, we give also results, simpler and partly known, for
three other models of random full binary trees: 
the extended binary search tree,
the critical beta-spltting random tree,
and
the uniform random full binary tree.
\end{abstract}

\maketitle

\section{Introduction}\label{S:intro}
In this paper, we study and compare fringe tree distributions for some
different types of random full binary trees.
The paper is inspired by the work by 
Aldous \cite{Aldous-clad,Aldous-beta2,Aldous-beta3,Aldous-beta1}
on random models for
phylogenetic trees, more precisely \emph{cladograms}. 
(These are for our purposes essentially the same as full binary trees; 
see \refR{Rclade1}.) 
Aldous \cite{Aldous-clad} introduced the \emph{beta-splitting model}
of a random full binary tree with a given number $n$ of leaves.
The model has a parameter $\gb\in [-2,\infty]$, and Aldous noted 
\cite[Section 4.1]{Aldous-clad}
that three special cases
give models that have appeared in many other applications (in, for example,
computer science); they are, with the notations used in the present paper
(see \refSS{SS1def} below for definitions):
\begin{description}
\item[$\beta = -3/2$] 
the uniform random full binary tree $\cU_n$.
\item[$\beta = 0$] 
the extended binary search tree $\eBST_{n-1}$,
\item[$\beta = \infty$] 
the random Patricia trie $\hTn$ (in the symmetric case $p=1/2$).
\end{description}
A fourth special case,
studied in, for example, \cite{Aldous-beta2, Aldous-beta3,Aldous-beta1}
is
\begin{description}
\item[$\beta = -1$] 
the critical beta-splitting random tree $\CB_n$;
\end{description}
this is also  very interesting because of its rich mathematically structure 
(and also fitting real phylogenetic trees well  in at least
some ways, see \cite[Sections 4.3--4.4]{Aldous-beta2}).

Fringe trees for several of these models have been studied before, 
see the references given in \refSS{SSfringe}.
We give in the present paper detailed results for these four models,
partly collected or adapted from earlier work,
and we complement these results with 
corresponding results for
the compressed binary search tree $\hBST_n$.

We consider asymptotics as the size  of the random tree tends to
infinity. 
For
a binary search tree (BST) $\BST_n$,
or a
random trie $\TR_n$, 
the asymptotic distribution of
a random fringe tree 
is given already by \citet{Aldous-fringe}.
More precise result including asymptotic normality of
the fringe tree counts
are proved in 
\cite{Devroye1991,Devroye2002} (BST) and
\cite{SJ347} (tries);
see also \eg{} \cite{FlajoletGM1997}, \cite{HwangN},
and
\cite{SJ296}.
These results can be used to derive corresponding results for 
Patricia tries \cite{Ischebeck}
and extended or compressed binary search trees, 
see \refSs{SPat}--\ref{SBST} below.
In particular, in these cases the fringe tree counts 
are asymptotically normal.

For the critical beta-splitting random tree $\CB_n$
the asymptotic  fringe tree distribution is found in \cite{Aldous-beta3},
see also \cite{Aldous-beta2}.
We improve earlier results and show convergence in probability of
(normalized) subtree counts 
(\refS{SCB}), 
but asymptotic normality has not yet been shown and remains an open problem.

For the uniform random full binary tree $\cU_n$,
the asymptotic tree fringe distribution was again found by \cite{Aldous-fringe}
(as a special case of 
results for more general conditioned Galton--Watson trees).
Also in this case, the fringe tree counts are asymptotically normal
\cite[Corollary 1.8]{SJ285}. 
This case is included below (\refS{Suni}) for comparison, but our results
are at most reformulations of known results and are not really new.

For Patricia tries $\hTn$, we find
explicitly the asymptotics of the mean and variance of 
fringe tree counts.
For the extended binary search tree $\eBST_n$, the asymptotic mean and variance 
of fringe tree counts
follow from known formulas for the binary search tree
$\BST_n$.
The asymptotic mean and variance of fringe tree counts for the uniform full
binary tree $\cU_n$ are also known.
However, for the compressed binary search tree 
$\hBST_n$, it seems more difficult 
to find the asymptotic mean and variance of  fringe tree counts.
We show how to find explicitly
the asymptotics of the mean, 
and thus the asymptotic distribution of a random fringe tree,
but this is done by a recursion and we cannot give a
general formula.
We leave as an open problem to find a formula for the asymptotic
variance of fringe tree counts for the compressed binary search tree.

In \refS{Sex} we 
give as examples explicit results for some small fringe trees.
We give there also tables with numerical values to faciliate comparison of the
five different models. One table includes for comparison also
numerical data, taken from 
\cite[Figure~7 and Section~5.3]{Aldous-beta3},
for a small sample of real cladograms.

\begin{remark}
Related results for fringe trees in other classes of random trees 
(with convergence in probability, and sometimes asymptotic normality)
have been given by many authors, see for example
\cite{Aldous-fringe}, 
\cite{BenniesK}, 
\cite{FengMahmoudP08}, 
\cite{Fuchs}, 
\cite{Chang}, 
\cite{DennertGrubel}, 
\cite{Fuchs2012}, 
\cite{GopMW}, 
\cite{SJ296}, 
\cite{SJ285}, 
\cite{SJ306}, 
\cite{SJ309}, 
\cite{SJ347}, 
\cite{SJ378}, 
and the further references therein.
\end{remark}

\subsection{The studied trees}\label{SS1def}
We use the following (common, but not universal) notation. 
(See \refSS{SSNot} for further notation.)
A \emph{binary tree} is a rooted tree such that each child of a node is
labelled either \emph{left} or \emph{right}, and each node has at most one
left and at most one right child.
A \emph{full binary tree} is a binary tree where each node has outdegree 0
or 2. (In the latter case, it thus has one left and one right child.)
A \emph{leaf} is a node with no children (i.e., outdegree 0).
The \emph{size} $|T|$ of a tree $T$ is its number of nodes,
and the \emph{leaf size}
$\esize{T}$ is its number of leaves.
All trees in the paper are non-empty, finite, and binary (and thus rooted),
except when we explicitly say otherwise.
We will mainly consider full binary trees, but note that binary trees are
not assumed to be full unless we say so.

It is well-known that a full binary tree has odd size, and that there is a
bijection between general binary trees of size $n\ge1$ and full binary trees
of size $2n+1$ defined as follows: 
Given a binary tree $T$, 
its \emph{extension} $\bT$
is the full binary tree obtained by
adding new leaves 
(often called \emph{external nodes})
at all
possible places, i.e., we add one new leaf to each node of outdegree 1, 
and two new leaves to each node of outdegree 0.
Conversely, the inverse map is: 
given a full binary tree, delete all its leaves. 
For many purposes, full and general binary trees are
thus equivalent, but both types are important, and there are several reasons
for studying properties of both classes of binary trees.

In the present paper we will also
study another relation between general binary trees and
full binary trees.
Given a binary tree $T$, its \emph{compression} $\hT$ is the full binary
tree obtained by deleting all nodes of outdegree 1 (and connecting the
remaining nodes in the obvious way).
This is obviously not a bijection; there is no way to reconstruct the binary
tree without further information. Note also that the 
size of the compressed tree typically is smaller; 
we have $1\le|\hT|\le|T|$, and every size in this range is possible;
however, $\hT$ and $T$ have the same number of leaves.

\begin{remark}\label{Rclade1}
  A \emph{cladogram} is a full binary tree with labelled leaves, 
where we do not 
care about the orientations, \ie, we do not
distinguish between left and right. (Formally, we may see a cladogram as an
equivalence classe of leaf-labelled full binary trees.)
In particular, the random full binary trees studied here may be regarded as
random cladograms by 
labelling the leaves (randomly, say) and
forgetting all orientations. 
Conversely, any cladogram may be regarded as a random leaf-labelled
full binary tree by randomly assigning orientations at all internal nodes.
It will be convenient for us to study full binary trees, but 
we can thus regard them as random cladograms, and mathematically this is
essentially equivalent. (At least in symmetric cases; for Patricia tries
with $p\neq\frac12$, the orientation is important.)
In particular, the results below on fringe trees yield as corollaries 
corresponding results for fringe trees of the corresponding random
cladograms
(these fringe trees are themselves random cladograms); we omit the
details. (See also \refR{Rclade2}.) 
\end{remark}

The random trees studied in the present paper are the following.

\subsubsection{Tries}
One well-known example of compression is for tries:
a \emph{trie} is a binary tree
constructed from a sequence of distinct infinite strings of 0 and 1
(see \refSS{SStries} for details), and
its compression is known as a \emph{Patricia trie};
see \cite[Section 6.3]{KnuthIII} for the computer science background.
We study in this paper tries and Patricia tries defined by independent
random strings where all bits are independent.
We denote the random trie defined by $n$ random strings by $\TR_n$,
and the corresponding Patricia trie by $\hTR_n$; these thus have $n$ leaves.

\subsubsection{Binary search trees}
The \emph{binary search tree (BST)} 
is another commonly studied random binary
tree. (See \refSS{SSBST} for definition.) 
We denote the random BST with $n$ nodes by $\BST_n$.
Just as tries and Patricia tries, it appears naturally in computer science
in connection with sorting and searching, see \eg{} 
\cite[Section 6.2.2]{KnuthIII}.
We will here study its extension, 
the \emph{extended binary search tree} $\eBST_n$ 
which is a full binary tree 
(with $n$ internal nodes and $n+1$ leaves).
The compressed version is perhaps less interesting in the computer-science
context, but 
from a 
mathematical perspective,
and as an analogue of Patricia tries,
we find it natural to also study the
\emph{compressed binary search tree},
which we denote by  $\hBST_n$.

\subsubsection{Uniform random full binary trees}
The \emph{uniform random full binary tree} $\cU_n$ is, as the name indicates, a
random tree sampled uniformly from the set of all full binary trees with $n$
leaves.

\subsubsection{Beta-splitting random trees}\label{SSSbeta}
The general definition of the \emph{beta-splitting random tree} with $n$ leaves
\cite{Aldous-clad} is that (if $n\ge2$)
the leaves are split randomly between the left
and right subtree at the root, such that the probability that the left
subtree has $i$ leaves is 
\begin{align}\label{beta}
  c(n;\gb)\frac{\gG(\gb+i+1)\gG(\gb+n-i+1)}{\gG(i+1)\gG(n-i+1)},
\qquad 1\le i\le n-1,
\end{align}
for a normalizing constant $c(n;\gb)$;
the construction proceeds then recursively in each subtree.
It can be checked that the cases $\gb=-3/2, 0, \infty$ (the latter
interpreted in a limiting sense) yield the random full binary trees 
$\cU_n$, $\eBST_{n-1}$, $\hTR_n$ (with $p=\frac12$) above
\cite[Section 4.1]{Aldous-clad}.
We consider here also
the case $\gb=-1$, known as the \emph{critical beta-splitting random tree},
which we denote by $\CB_n$
(it is denoted $\DTCS(n)$ in \cite{Aldous-beta2,Aldous-beta3}); 
it is thus defined with the splitting probabilities
\eqref{beta} being
\begin{align}\label{beta-1}
 \frac{n}{2h_{n-1}}\cdot\frac{1}{i(n-i)}
=
 \frac{1}{2h_{n-1}}\Bigpar{\frac{1}{i}+\frac{1}{n-i}},
\qquad 1\le i\le n-1,
\end{align}
where we use  the harmonic numbers
\begin{align}\label{harmonic}
h_n:=\sum_1^n\frac1i.   
\end{align}

\subsection{Fringe trees}\label{SSfringe}
In this paper we focus on  fringe trees of the various random trees.
 For a rooted tree $T$ and a node $v\in T$, 
the \emph{fringe tree} $T^v$
is the subtree of $T$ consisting of $v$ and all its descendants; 
the fringe tree $T^v$ is
a rooted tree with root $v$, and if $T$ is a binary tree or a full
binary tree, then so is each of its fringe trees.
If $\tT$ and $t$ are two binary trees,
let $N_{t}(\tT)$ be the number of fringe trees of $\tT$ that are equal to
$t$ (in the sense of isomorphic as binary trees), i.e.,
\begin{align}\label{tor1}
  N_{t}(\tT) := \bigabs{\set{v\in \tT:\tT^v=t}}.
\end{align}
Here $t$ will always be a fixed tree (think of it as small),
while $\tT$ usually will be a (big) random tree; then $N_{t}(\tT)$ is a
random variable.
We consider also the random fringe tree $\Tx$ defined as $T^v$ for a 
node $v\in T$ chosen uniformly at random.
(If $T$ is random, we first condition on $T$ and then choose a node $v$
in it.)
Thus, for every fixed binary tree $t$ and a 
non-random binary tree $T$,
\begin{align}\label{tor3}
\P(\Tx=t)= \frac{ N_{t}(T)}{|T|}.
\end{align}
For a random binary tree $T$, \eqref{tor3} holds conditionally on $T$,
and thus
\begin{align}\label{tor2}
\P(\Tx=t)= \E\frac{ N_{t}(T)}{|T|}.
\end{align}
See \cite{Aldous-fringe} for a general study of fringe tree distributions,
including explicit results for several classes of random trees.

In this paper we discuss results for the random fringe tree $\Tx_n$ of a
sequence of random trees $T_n$.
We state results both conditioned on the tree $T_n$ 
and unconditioned, and
we use the standard  terminology and call 
such results \emph{quenched} and \emph{annealed}, repectively.

\begin{remark}
We may also care only about the leaf size
of the fringe trees, and define, for $m\ge1$,
\begin{align}
  \label{tor10}
  N_m(T) := \bigabs{\set{v\in T:\esize{T^v}=m}}
=\sum_{t:\esize{t}=m}N_t(T)
.\end{align}
Let $\cX_n$ denote one of the random trees $\hTR_n, \eBST_{n-1}, \cD_n$
studied here (which all have $n$ leaves).
Then
the recursive construction of $\cX_n$ implies  that conditioned on some fringe
tree $\cX_n^v$ having $m$ leaves, that fringe tree has the same distribution as 
the random tree $\cX_m$ of the same type with $m$ leaves.
The same holds evidently also for $\cX_n=\cU_n$.
In particular, we have for any full binary tree $t$ with $m$ leaves
\begin{align}
  \label{tor11}
\E [N_t(\cX_n)] = \P(T_m=t)\E[N_m(\cX_n)].
\end{align}
\end{remark}

\section{Preliminaries}\label{Sprel}

\subsection{Notation}\label{SSNot}
(See also \refSS{SS1def}; we allow some repetitions below for clarity.)
For a rooted tree $T$, the set of leaves of $T$ is denoted $\cL(T)$.
The leaves are also called \emph{external nodes} and the other nodes, \ie,
those with outdegree $>0$, are called \emph{internal nodes}.
The number of nodes of $T$ is denoted by $|T|$
(the \emph{size} of $T$), 
the number of leaves
(external nodes) by
$\esize{T}:=|\cL(T)|$
(the \emph{leaf size} of $T$),
and the number of internal nodes by
$\isize{T}:=|T|-\esize{T}$.
Recall that in a binary tree $T$, 
the number of nodes of outdegree 2 is $\esize{T}-1$, and thus 
the number of nodes of outdegree 1 is $|T|-2\esize{T}+1$.
Hence, if $T$ is a full binary tree, then 
$\isize{T}=\esize{T}-1$ and 
\begin{align}\label{full}
|T|=2\esize{T}-1.  
\end{align}

The \emph{root degree} $\rho(T)$ is the (out)degree of the root $o\in T$.

We let $\TI$ denote the tree consisting of a root only, so
$|\TI|=\esize{\TI}=1$. 

Let 
$\fT$ be the set of all  binary trees, and
$\hfT$  the subset of all full  binary trees, and
$\fTp$ the set of all binary trees such that no leaf has a parent of
outdegree 1;
thus $\hfT\subset \fTp\subset\fT$.
Furthermore, for any set $S\subseteq\set{0,1,2}$, let 
$\fTsett{S}:=\set{T\in\fT:\rho(T)\in S}$
be the set of all binary trees with root degree in $S$.
In particular, $\fTset0=\set{\TI}$.

If $t$ is a full binary tree, let
$\cht$ be the set of all binary trees that can be obtained from $t$ by
subdividing the edges, \ie, by
replacing every edge by a path of $\ell\ge1$ edges;
each such path thus contains $\ell-1$ new nodes of outdegree 1. 
Note
that each new node has its (only) child as either a left or a right child.
Let further $\chtt$ 
be the set of all binary trees that compress to $t$; 
note that this is larger than $\cht$ since
$\chtt$ allows also adding a path from the root to a new 
root, but $\cht$ does not; in fact, $\cht=\chtt\cap\fTset{0,2}$.
(If $t\neq\TI$, then $\cht=\chtt\cap\fTset{2}$.)

If $v$ is a node in a binary tree $T$, then the \emph{left depth} $\ld(v)$
is the number of edges 
that go from some node to its left child
in the path from the root to $v$.
The \emph{right depth} $\rd(v)$ is defined similarly. 
Note that $\ld(v)+\rd(v)=d(v)$, the depth of $v$.
When necessary, we write $d_T(v)$ for the depth in $T$.

We define the \emph{left external path length} $\LPL(T)$
and \emph{right external path length} $\RPL(T)$ by
\begin{align}\label{LPL}
  \LPL(T) = \sum_{v \in\cL(T)}\ld(v),
\qquad
  \RPL(T) = \sum_{v \in\cL(T)}\rd(v).
\end{align}
Thus the sum $\LPL(T)+\RPL(T)$ equals the total external path length
$\sum_{v\in\cL(T)}d(v)$. 

The fringe tree $T^v$ and the fringe tree counts $N_{t}(T)$
are defined in the introduction.

We use  $\dto$ and $\pto$ to denote convergence in
distribution and probability, respectively, of random variables.
We further say that   \emph{$X_n\dto Y$ with all moments}
if $X_n\dto Y$  and also $\E [X_n^r]\to \E [Y^r]$ for every integer $r>0$.
We let $\op(1)$ denote any sequence of random variables $X_n$ such that
$X_n\pto0$.

$\Law(X)$ denotes the distribution of a random variable $X$.

$N(\mu,\gss)$ denotes the normal distribution with mean $\mu$ and variance
$\gss\ge0$. 
$\Po(\gl)$ denotes the Poisson distribution with parameter $\gl\ge0$.
We thus have 
\begin{align}
  \label{po}
\Po(\gl;n):=\Po(\gl)(n) = \frac{\gl^n}{n!}e^{-\gl},
\qquad n=0,1,\dots
\end{align}

We may sometimes abbreviate ``uniformly random'' to ``random''.
$\log$ denotes natural logarithms.
Unspecified limits are as \ntoo.

\subsection{Tries}\label{SStries}

A \emph{trie}
(see \eg{} \cite[Section 6.3]{KnuthIII} and \cite[Section 1.4.4]{Drmota})
is a rooted tree constructed from a set of $n\ge1$ distinct
strings $\Xi\sss1,\Xi\sss2,\dots,\allowbreak\Xi\sss{n}$ 
in some alphabet $\cA$;
we consider here only the case $\cA=\setoi$, and then
the trie will be a binary tree.
(Fringe trees for tries defined by arbitrary finite alphabets are treated by
\citet{Ischebeck}.)
We assume for convenience that the strings are infinite;
thus $\Xi\sss{i}=\xi\sss i_1\xi\sss i_2\dotsm\in\setoi^\infty$ for every $i$.
The trie is constructed recursively. 
If $n=0$, then the trie  is the empty tree $\emptyset$.
Otherwise,
we begin with a root, and put every string in the root. 
If $n=1$, then we stop there, so the trie equals $\TI$.
Otherwise, \ie, if $n\ge2$, we pass all strings
to new nodes; 
all strings beginning with 0 (if any) are passed to a left child of the root
and 
all strings beginning with 1 (if any) are passed to a right child of the root.
We continue recursively, the next time partitioning the strings according to
the second letter, and so on, always looking at the first letter not yet
inspected.
At the end there is a tree with $n$ leaves, each containing one string. 
Equivalently, each string $\Xi\sss i$ defines an infinite path $\gG_i$
from the root
in the infinite binary tree, by processing the letters in the string in
order and going to a left child for every 0 and a right child for every 1.
We then stop each path at first node that it does not share with any other
of the paths $\gG_j$; these nodes are the leaves of the trie, and the
trie is the union of the $n$ stopped paths.
Note that in a trie, a parent of a leaf must have outdegree 2, \ie, a trie
belongs to the set $\fTp$.

We will consider the random trie $\TR_n$ defined by $n$ random strings
$\Xi\sss1,\dots,\Xi\sss{n}$ where we assume that the strings are
independent, and furthermore in each string the letters are independent
and identically distributed with distribution $\Be(p)$ for some $p\in(0,1)$,
\ie, each letter $\xi\sss{i}_k$ has $\P(\xi\sss{i}_k=1)=p$ and
$\P(\xi\sss{i}_k=0)=q:=1-p$.
We omit the parameter $p$ from the notation, but it is implicit when we
discuss tries; we use always the notations $p$ and $q=1-p$ in the sense above.

It is well-known that many results for tries show (typically small) periodic
oscillations instead of limits as \ntoo,
see \eg{}
\cite{
KnuthIII, 
Mahmoud, 
FlajoletRV, 
FHZ2010, 
SJ242, 
FHZ2014, 
JS-Analytic, 
SJ347}.
More precisely, if $\log p/\log q$ is irrational, then such oscillations do
not occur, but if $\log p/\log q$ is rational they typically do.
We call the case when $\log p/\log q\in\bbQ$ \emph{periodic}; 
otherwise we have the \emph{aperiodic} case.
In the periodic case, if $\log p/\log q$ equals $a/b$ in lowest terms 
(with $a,b\in\bbN$), then define 
\begin{align}\label{dp}
  d =
\dpp:=\frac{-\log p}a=\frac{-\log q}b
>0, 
\end{align}
the greatest
common divisor of $-\log p$ and $-\log q$.

The Patricia trie $\hTR_n$ is obtained by compressing $\TR_n$.
Note that the trie $\TR_n$ and the Patricia tree $\hTR_n$ both have exactly
$n$ leaves: 
\begin{align}
\esize{\TR_n}=\esize{\hTR_n}=n.  
\end{align}
The Patricia tree $\hTR_n$ thus has $n-1$ internal nodes, while
the number of internal nodes in $\TR_n$ is random.

\subsection{Binary search trees}\label{SSBST}

Binary search trees may be constructed in different (but equivalent) ways;
we will use the following, closely connected to the sorting algorithm
\emph{Quicksort}
(see \eg{} \cite[Section 6.2.2]{KnuthIII} and
\cite[Section 1.4.1]{Drmota}):
Consider a set of $n$ distinct items, 
which we may assume are real numbers $x_1,\dots,x_n$. 
If $n=0$, the BST $\BST_0$ is the empty tree $\emptyset$, and if $n=1$,
$\BST_1:=\TI$. 
If $n\ge 2$, pick one of the $n$ items at random, and call it the
\emph{pivot}. 
Compare all other elements to the pivot, and let $L$ be the set of all $x_i$
that are smaller than the pivot, and let $R$ be the set of all items
$x_i$ that are greater than the pivot. (Thus, $|L|+|R|=n-1$.)
The BST $\BST_n$ is defined as 
the binary tree with the root having left and right
subtrees that are
constructed recursively from the sets $L$ and $R$, respectively.
It is easily seen by induction that $|\BST_n|=n$; in fact, it is natural
to label the root by the pivot, and then during the recursion
each node becomes labelled by
exactly one of the $n$ numbers $x_i$.
Note that in the construction, $|L|$ is uniformly distributed over the
$n$ numbers \set{0,1,\dots,n-1}.
Hence \cite[Section 4.1]{Aldous-clad},
by comparing with the definition above of the beta-splitting random tree
and noting that for $\gb=0$, \eqref{beta} yields the uniform distribution on
$\set{1,\dots,n-1}$, we see by induction that the beta-splitting random tree
with $n$ leaves and $\beta=0$ equals the \emph{extended binary search tree}
$\eBST_{n-1}$ defined by adding $n$ leaves (external nodes) to $\BST_{n-1}$.

We will also, for comparison with the other full binary trees studied here,
consider the \emph{compressed binary search tree} $\hBST_n$, defined (as
said in the introduction) by removing (contracting) all nodes of outdegree 1.
Note that both $\esize{\BST_n}=\esize{\hBST_n}$ and $|\hBST_n|$ are random.

It is shown by
\citet[Example 3.3]{Aldous-fringe} and
\citet[Theorem 2]{Devroye1991},
that, as \ntoo, 
\begin{align}\label{jw1}
\esize{\BST_n}/n=\esize{\hBST_n}/n\pto 1/3
\end{align}
and thus, see \eqref{full},
 \begin{align}\label{jw2}
  |\hBST_n|/n\pto 2/3. 
\end{align}
More precisely   \cite{Devroye1991},
\begin{align}\label{jw3}
  \frac{\esize{\BST_n}-n/3}{\sqrt{n}}
=  \frac{\esize{\hBST_n}-n/3}{\sqrt{n}}
\dto N(0,2/45)
\end{align}
and thus  $ \frac{\abs{\hBST_n}-2n/3}{\sqrt{n}}\dto N(0,8/45)$
by \eqref{full};
see also 
\cite{Mahmoud1986} for expectations and
\eg{} \cite[Theorem 6.9]{Drmota} for a more general result.

\subsection{Fringe trees of compressed trees}\label{SSfringe2}
Let $T$ be a binary tree and $\hT$ its compression.
Note first that the leaves of $\hT$ are precisely the leaves of $T$,
while the roots may differ. (In general, there may in $T$
be a path of nodes of outdegree 1 from the root of $T$ to the root of
$\hT$.)  Thus, $\esize{\hT}=\esize{T}$.

Let $t$ be a full binary tree and consider the nodes $v\in \hT$ such that
$\hT^v=t$.  We consider two cases separately:
\begin{romenumerate}
\item   
If $t=\TI$, then:
$v\in\hT$ and $\hT^v=t$
$\iff$ $v$ is a leaf in $\hT$
$\iff$
$v$ is a leaf in $T$
\\$\iff$ $v\in T$ and $T^v=t$.

\item 
If $|t|>1$, then 
the root of $t$ has outdegree 2, and:\\
$v\in\hT$ and $\hT^v=t$ 
$\iff$ $v$ has outdegree 2 in $T$ and 
$t$ is the contraction of $T^v$.
\end{romenumerate}
Consequently, in both cases,
\begin{align}
  \label{fr1}
v\in\hT \text{ and } \hT^v=t \iff  v\in T \text{ and } T^v\in \cht, 
\end{align}
where we recall that
$\cht$ is the set of all binary trees that can be obtained from $t$ by
subdividing the edges.

Define the functional $\gf$ on binary trees by 
\begin{align}
  \label{od0}
\gf(T):=\indic{T\in\cht},
\end{align}
and define the corresponding additive functional
\begin{align}\label{od1}
  \Phi(T):=\sum_{v\in T}\gf(T^v).
\end{align}
Then, by \eqref{fr1},
\begin{align}\label{od2}
  N_{t}(\hT) 
=\sum_{v\in\hT}\indic{\hT^v=t}
=\sum_{v\in T}\indic{T^v\in\cht}
=\sum_{v\in T}\gf(T^v)
= \Phi(T).
\end{align}

\subsection{Extended fringe trees}\label{SSext}
Let $T$ be a rooted tree.
The fringe tree $T^v$ defined in the introduction consists of a node $v$ and
all its descendants.
\citet{Aldous-fringe} introduces also the \emph{extended fringe tree} by 
also going up from the chosen node $v$ to ancestors and then taking
descendants, thus including siblings,
cousins, and so on; we may formally 
define the extended fringe tree as the nested sequence
of fringe trees $(T^{v_i})_{i=0}^{d(v)}$ where $v_i$ is the $i$th ancestor of $v$.
(See \cite[Section 4]{Aldous-fringe} for details.)
It is shown in \cite[Proposition 11]{Aldous-fringe} that if the 
random fringe trees of some sequence of random trees $T_n$ converge in
distribution,
then so do the random extended fringe trees.
The limit then is an infinite nested sequence of random trees 
$(T_\infty\sss i)_{i\ge0}$.

\begin{remark}\label{Rsin}
 Assume that these limits 
$T_\infty\sss i$
exist and
also the technical condition (satisfied for the random trees
studied here) that the depth of a random node in $T_n$ tends to infinity
in probability.
Then the random limit sequence 
$(T_\infty\sss i)_{i\ge0}$
can be combined to a random  infinite tree,
called \emph{sin-tree},
with a single infinite path from $o$ consisting of
the roots of the trees $T_\infty\sss i$, see \cite{Aldous-fringe}.
The results below can be translated to this, more intuitive, description of
the limit of the extended fringe tree.
\end{remark}

We consider in the present paper, as in \cite{Aldous-beta2,Aldous-beta3},
the extended fringe tree 
$\TXX=(\TXXX{i})_i$ 
of a uniformly random \emph{leaf} in the tree $T$
(instead of a random node as for $\Tx$); 
this also converges in 
distribution when the random fringe tree does, since it can be regarded as
the random extended fringe tree conditioned on the fringe tree 
$\Tx=\TXXX0$
being $\TI$.

\begin{remark}
$\TXXX{i}$ is undefined for $i$ greater than the depth of the chosen random
leaf; this is no problem for us since for the random trees studied here, 
for every fixed $i$, the probability that the depth of a random leaf is 
at most $i$
tends to 0; hence $\TXXX{i}$ is defined with probability tending to 1, which
is all we need.
Cf.~\refR{Rsin}.  
\end{remark}

It is easy to see that
if $T$ is any (deterministic or random) tree, 
$(t,\ell)$ is a pair of a tree $t$ and a marked leaf $\ell\in t$,
$i\ge0$,
and $o$ is the root of $\TXX$ (defined as the root of $\TXXX0$), then
\begin{align}\label{txx1}
  \P\bigpar{(\TXXX i,o)=(t,\ell)} = q(T;t,\ell)\indic{ i=d_t(\ell)}
,\end{align}
where, letting $L$ be a uniformly random leaf in $T$,
\begin{align}\label{qsinl1}
\qsin(T;t,\ell):=\P\bigpar{\exists w\in T: L\in T^w \text{ and } (T^w,L) 
\text{ is isomorphic to } (t,\ell)}.
\end{align}
We define also the simpler
\begin{align}\label{qsin1}
\qsin(T;t):=\P\bigpar{\text{a random leaf $v$ lies in some fringe tree 
$T^w$ isomorphic to $t$}}.
\end{align}
Note that, if $|T|>1$, any leaf lies in several fringe trees of different
sizes; hence $\sum_t \qsin(T;t)>1$ 
so $\qsin(T;\cdot)$ is not a probability distribution on trees.
In fact, trivially $\qsin(T;\TI)=1$ for every tree $T$.

For a random tree $T$ we define also $\qsin(T;t,\ell\mid T)$ 
and $\qsin(T;t\mid T)$ 
by \eqref{qsinl1} and \eqref{qsin1}
conditioned on $T$.
Thus $\qsin(T;t\mid T)$ is a functional of $T$, and thus a random variable, 
while $\qsin(T;t)$ is a  number depending on the distribution of $T$ only,
and similarly for $\qsin(T;t,\ell)$.
We have, as always for conditional expectations,
\begin{align}\label{qsin00}
  \qsin(T;t,\ell) = \E[\qsin(T;t,\ell\mid T)],
\qquad
  \qsin(T;t) = \E[\qsin(T;t\mid T)].
\end{align}

Let $t$ be a fixed tree with a marked leaf $\ell$, and let $T$ be a
(deterministic or random) tree.
For every fringe tree $T^w$ that is isomorphic to $t$, 
there is exactly one leaf $L$ in $T$ such that $L\in T^w$ and
$(T^w,L)=(t,\ell)$.
Since all copies of $t$ in $T$ are disjoint, these leaves $L$ are distinct
for different fringe trees $T^w=t$, and thus (for a random tree $T$)
\begin{align}\label{qsin2l}
  \qsin(T;t,\ell\mid T)=\frac{N_t(T)}{\esize{T}}
\end{align}
and hence 
\begin{align}\label{qsin3l}
  \qsin(T;t,\ell)=\E\frac{N_t(T)}{\esize{T}}.
\end{align}
Consequently, recalling \eqref{txx1}, the distribution of the random
extended fringe tree $\TXX$ is determined by the fringe subtree counts
$N_t(T)$, and conversely.
More precisely, the quenched (i.e., conditional) distribution of $\TXX$ is
determined by the (random) counts $N_t(T)$, and the annealed distribution
by their expectations.)
Furthermore, again since the copies of $t$ in $T$ are disjoint, the
number of leaves that lie in some copy of $t$ equals $N_t(T)\esize{t}$,
and hence,
for any (random) tree $T$,
\begin{align}\label{qsin2}
  \qsin(T;t\mid T) = \frac{N_t(T)\esize{t}}{\esize{T}}
\end{align}
and 
\begin{align}\label{qsin3}
  \qsin(T;t) = \esize{t}\E\frac{N_t(T)}{\esize{T}}
.\end{align}
Consequently, for every leaf $\ell\in t$,
\begin{align}\label{qsin4}
  \qsin(T;t,\ell\mid T) = \frac{1}{\esize{t}}  \qsin(T;t\mid T),
\qquad   \qsin(T;t,\ell) = \frac{1}{\esize{t}}  \qsin(T,t).
\end{align}
(This follows also directly from \eqref{qsinl1}--\eqref{qsin1}, noting that
by symmetry $\qsin(T;t,\ell)$ does not depend on the leaf $\ell\in T$.)
Hence, it suffices to study the simpler $\qsin(T;t\mid T)$ and $\qsin(T;t)$
without marked leaves.

For a sequence of (deterministic or random)  trees $T_n$,
it follows from \eqref{txx1} and \eqref{qsin4} that
the extended fringe trees $\TXX_n=(\TXXX{i}_n)_i$ converge in distribution, to
some sequence of (random) trees $\TXXoo=(\TXXXoo i)_{i\ge0}$,
if and only if $q(T_n;t)$ converges for every fixed tree $t$,
and in this case the limit distribution is determined by the limits
\begin{align}\label{qsin5}
q(\TXXoo;t)
:=\lim_\ntoo q(T_n;t)
.\end{align}
More precisely, if these limits exist, then 
it follows from \eqref{txx1} and \eqref{qsin4} that 
if $o$ is the root of $\TXXoo$ (i.e., the root of $\TXXXoo0$),
then
for every tree $t$ and leaf $\ell\in t$ with $d_t(\ell)=i$,
\begin{align}\label{qsin6}
  \P\bigpar{(\TXXXoo{i},o)=(t,\ell)} 
=\frac{1}{\esize{t}}
\lim_\ntoo q(T_n;t)
=\frac{1}{\esize{t}}
q(\TXXoo;t)
.\end{align}

Note also that for any (deterministic) tree $T$, \eqref{qsin2} and \eqref{tor3}
yield
\begin{align}\label{qsin7}
  \qsin(T;t)
=\frac{\esize{t}|T|\P(\Tx=t)}{\esize{T}}
=\esize{t}\frac{|T|}{\esize{T}}\P(\Tx=t)
=\esize{t}\frac{\P(\Tx=t)}{\P(\Tx=\TI)}
.\end{align}
This leads to the following version of the result by
\citet[Proposition 11]{Aldous-fringe} that,
as said above, 
if a sequence of random fringe trees converge in distribution, 
then so do the extended fringe trees. 
(Note that we here, unlike \cite{Aldous-fringe}, consider extended fringe
trees rooted at a random leaf; the fringe trees $\Tx_n$ are (of course)
rooted at a random node.)
We also obtain a formula for the limits \eqref{qsin5}.
\begin{lemma}\label{Lfringe}
Let $(T_n)$ be a sequence of random trees such that, as \ntoo, 
\begin{align}
  \label{qsin20}
\P(\Tx_n=t\mid T_n) \pto \P(\Txoo=t)
\end{align}
for every fixed tree $t$ and
some random tree $\Txoo$ 
(a quenched limiting fringe tree).
Then
\begin{align}\label{qsin12}
  \qsin(T_n;t\mid T_n)
\pto
q(\TXXoo;t)
= \kk\esize{t}{\P(\Txoo=t)}
,\end{align}
where (with the limit in probability, in general)
\begin{align}\label{qsin13}
  \kk:=\frac{1}{\P(\Txoo=\TI)}=\lim_\ntoo\frac{|T_n|}{\esize{T_n}}.
\end{align}
In other words, the (quenched) distribution of the extended fringe tree
converges in probability: 
\begin{align}\label{qsin21}
\Law(\TXX_n\mid T_n) =\Law((\TXXX{i}_n)_{i\ge0}\mid T_n)
\pto \Law(\TXXoo)=\Law((\TXXXoo i)_{i\ge0})
,\end{align}
where the distribution of the limiting extended fringe tree $\TXXoo$ is
given by $\qsin(\TXXoo;t)$ in \eqref{qsin12}.

Conversely, if the extended fringe tree distribution converges
in the sense \eqref{qsin21}, and the limit $\kk$ in \eqref{qsin13} exists,
then the fringe tree distribution converges in the sense \eqref{qsin20},
and \eqref{qsin12} holds.

For full binary trees with $|T_n|\pto\infty$, 
we always have \eqref{qsin13} with $\kk=2$.
\end{lemma}
\begin{proof}
We have by \eqref{qsin7} and the assumption \eqref{qsin20}
\begin{align}\label{qsin11}
  \qsin(T_n;t\mid T_n)
=\esize{t}\frac{\P(\Tx_n=t\mid T_n)}{\P(\Tx_n=\TI\mid T_n)}
\pto
\esize{t}\frac{\P(\Txoo=t)}{\P(\Txoo=\TI)}
,\end{align}
which can be written as \eqref{qsin12}--\eqref{qsin13},
since \eqref{qsin20} as a special case also yields
\begin{align}
  \frac{\esize{T_n}}{|T_n|} =\P(\Tx_n=\TI\mid T_n)
\pto \P(\Tx_\infty=\TI).
\end{align}
Furthermore, \eqref{qsin12} is equivalent to \eqref{qsin21}  
by the discussion before \eqref{qsin5}.

The converse follows in the same way.

For full binary trees with $|T_n|\pto\infty$,
\eqref{full} yields $\kk=2$.
\end{proof}

Results for extended fringe trees thus follow rather trivially from results
for fringe trees $\Tx_n$, but we find it interesting to state also such
results explicitly below.

\section{Patricia tries}\label{SPat}

Recall that the random trie $\TR_n$ and its contraction the Patricia trie
$\hTR_n$ are defined using a parameter $p\in(0,1)$, which is kept fixed and
is omitted from the notation, see \refSS{SStries}.

We fix a full binary tree $t$ and consider $N_t(\hTR_n)$, 
the number of fringe trees in the  Patricia trie $\hTR_n$ that equal $t$.
Let $m:=\esize{t}$, the number of leaves in $t$.
The case $m=1$ is trivial with $t=\TI$ and $N_t(\hTR_n)=n$, so we assume
$m\ge2$.

As shown by \citet{Ischebeck},
asymptotic normality of $N_t(\hTR_n)$ follows by \eqref{od2} and a
straightforward application of results for tries in \cite{SJ347}, see below
for details, although some work is required to calculate the asymptotic mean
and variance. 

We first introduce some notation.
Let,
recalling the definition of $\cht$ in \refSS{SSNot},
\begin{align}\label{pit}
\pi_t :=  \P(\TR_{m}\in \cht).
\end{align}
(This is thus the probability that the trie $\TR_m$ compresses to $t$ and
has root degree $\neq1$.) 
This will be calculated in \refL{LPT} below.
Further, for $\gl>0$, 
let $\Tgl$ be the random trie constructed from a random number
$N_\gl\in\Po(\gl)$ random strings; thus $\Tgl$ has $N_\gl$ leaves.
The results in \cite{SJ347} and \cite{Ischebeck}
use heavily some functions defined in general
(assuming $\gf(\TI)=0$ as in our case) by
\cite[(3.16)--(3.18)]{SJ347}:
  \begin{align}
    \fE(\gl)&:=\E\gf(\Tgl),\label{fE0}
    \\
    \fV(\gl)&:=2\Cov\bigpar{\gf(\Tgl),\GF(\Tgl)}-\Var\gf(\Tgl),\label{fV0}
              \\
    \fC(\gl)&:=
\Cov\bigpar{\gf(\Tgl),N_\gl},
\label{fC0}
 \end{align}
and their Mellin transforms defined by (for $\sX=\sE,\sV,\sC$ and suitable $s$)
\begin{align}\label{mfx}
  \MfX(s):=\intoo \fX(\gl)\gl^{s-1}\dd\gl.
\end{align}
For $\sX=\sE,\sV,\sC$ define also the function $\psiX(x)$, $x\in\bbR$ by
\cite[(3.13)--(3.14)]{SJ347}:
\begin{romenumerate}
\item 
In the aperiodic case ($\log p/\log q\notin\bbQ$),
$\psiX$ is constant: for all $x$,
\begin{align}\label{psiX0}
\psiX(x):=\Mx{\fX}(-1).  
\end{align}
\item 
In the periodic case, 
$\psiX$ is a continuous
$d$-periodic function given by the Fourier series,
recalling $d=\dpp$ defined in \eqref{dp},
\begin{equation}\label{psiX}
  \psiX(x)=\sumkoooo  \Mx{\fX}\Bigpar{-1-\frac{2\pi k}{d}\ii}e^{2\pi\ii  kx/d}.
\end{equation}
(The Fourier series converges absolutely in our case by
\eqref{MfE}--\eqref{MfC} below, or by \cite[Theorem 3.1]{SJ347} and the
formulas for $f_X(\gl)$ below.)
\end{romenumerate}

Finally, let
\begin{align}\label{H}
  H:=-p\log p - q\log q,
\end{align}
the entropy of the bits in the random strings $\Xi\sss{i}$ used to define
$\hTR_n$. 
 
\begin{theorem}[Partly \citet{Ischebeck}]\label{TP}
Let $t$ be a full binary tree with $\esize{t}=m>1$.
Then $N_t(\hTR_n)$ is asymptotically normal as \ntoo:
\begin{align}
\frac{N_t(\hTn)-\E N_t(\hTn)}{\sqrt{\Var N_t(\hTn)}}
  &\dto N\xpar{0,1},\label{ttn}
\end{align}
with convergence of all moments.
We have  $\Var N_t(\hTn)=\Theta(n)$ as \ntoo, and
\begin{align}
  \label{ttmeann}
\E N_t(\hTn)&=n H\qw\psiE(\log n)+o(n),
\\
\Var N_t(\hTn)&
=n\bigpar{H^{-1}\psiV(\log n)-H^{-2}\psiC(\log n)^2 +o(1)}
\label{tt+hgs}
,\end{align}
where $\psiX$ is given by \eqref{psiX0}--\eqref{psiX} with
(for $\Re s>-m$)
\begin{align}\label{MfE}
\MfE(s) &= \pi_t\frac{\gG(m+s)}{m!},
 \\\label{MfV}
\MfV(s)&
=\frac{\pi_t}{m!}\gG(m+s)
-\frac{\pi_t^2}{m!^2}2^{-2m-s}\gG(2m+s)
\notag\\&\qquad
-2\frac{\pi_t^2}{m!^2}
\sumko(-1)^k\frac{\gG(2m+s+k)}{k!}\cdot
\frac{p^{m+k}+q^{m+k}}{1-(p^{m+k}+q^{m+k})}
,\\\label{MfC}
\MfC(s)&
=-s\MfE(s)
= -\pi_t \frac{s\gG(m+s)}{m!}.
\end{align}
In the aperiodic case,
\eqref{ttmeann}--\eqref{tt+hgs}
simplify to
\begin{align}
  \label{ttmeann0}
\E N_t(\hTn)/n&\to H\qw\frac{\pi_t}{m(m-1)},
\\
\Var N_t(\hTn)/n& \to
H^{-1}\Mx{\fV}(-1) -
\Bigpar{H\qw\frac{\pi_t}{m(m-1)}}^2
\label{tt+hgs0}
.\end{align}
\end{theorem}

\begin{proof}
The asymptotic normality \eqref{ttn} is proved by 
\cite[Theorem 1 and Remark 4]{Ischebeck}, 
but the formulas \eqref{MfE}--\eqref{tt+hgs0} are
not calculated  explicitly there, so we will show them here. 
(The calculations are similar to the calculations in \cite{Ischebeck}  for
$N_m(\hTR_n)$; in particular,  recall the relation \eqref{tor11}.)

First, for completeness, we sketch (in our notation)
the proof of \eqref{ttn}--\eqref{tt+hgs} in
\cite{Ischebeck} based on \cite{SJ347}.
By \eqref{od2}, we have $N_t(\hTn)=\GF(\Tn)$, and 
we apply \cite[Theorem 3.9]{SJ347} to $\GF(\Tn)$.
We first verify the technical condition there that we can write
$\gf=\gf_+-\gf_-$ with bounded $\gf_\pm$ such that the corresponding
additive functionals $\Phi_\pm$ are increasing, in the sense that
$\Phi_\pm(T_1)\le\Phi_\pm(T_2)$ if $T_1$ is a subtree of $T_2$.
As in \cite[Section 4.3]{SJ347}, we let
$\gf_{>m}(T):=\indic{\esize{T}>m}$, 
and it is easily seen that $\gf_+:=\gf+\gf_{>m}$ and $\gf_-:=\gf_{>m}$
satisfy the condition.
Hence \cite[Theorem 3.9]{SJ347} applies.
Furthermore, $\gf(\Tn)=0$ for $n>m$, and it is easy to see that
there exists $n$ such that $\Var \GF(\Tn)>0$ (\ie, $\GF(\Tn)$ is not constant);
we may for example take $n=m+1$ if $m\ge 3$ and $n=4$ for $m=2$.
Hence also \cite[Lemma 3.14]{SJ347} applies, which shows that
$\Var N_t(\hTn)=\Var\GF(\Tn)=\gO(n)$ as \ntoo.
Consequently, \eqref{ttn} (with convergence of moments)
follows by \cite[Theorem 3.9(iv)]{SJ347}.
Furthermore, the moment convergence in
\cite[Theorem 3.9(ii)]{SJ347} implies
\eqref{tt+hgs}. In particular, 
$\Var\GF(\Tn)=O(n)$, and thus
$\Var\GF(\Tn)=\Theta(n)$.
The asymtotics \eqref{ttmeann} follows from \cite[Theorem 3.9(v)]{SJ347}.

It remains to find the Mellin transforms $\MfX$ (which yield $\psiX(x)$
by \eqref{psiX0}--\eqref{psiX}).
First,
since $\TR_n\in\cht$ is possible only when $n=\esize{t}=m$,
it follows by \eqref{od0} and \eqref{pit} that 
\begin{align}\label{pat2}
  \fE(\gl) 
:=\E \gf(\Tgl)
=\P(\Tgl\in\cht)
= \P(N_\gl=m)\pi_t
=\Po(\gl;m)\pi_t
=\pi_t\frac{\gl^m}{m!}e^{-\gl}
.\end{align}
This, or simpler \cite[Lemma 3.16]{SJ347}, yields the Mellin transform
\begin{align}\label{pat3}
  \MfE(s) :=\intoo \fE(\gl)\gl^{s-1}\dd\gl = \frac{\gG(m+s)}{m!}\pi_t,
\end{align}
verifying \eqref{MfE}.
By \cite[Lemma 3.6]{SJ347},  we have
$\MfC(s)=-s\MfE(s)$, showing \eqref{MfC}.
(We also obtain  $\fC(\gl)=\gl\fE'(\gl)=(m-\gl)\fE(\gl)$.)

To find the more complicated $\fV$, note first that if $\gf(T)=1$, \ie,
$T\in\cht$, then no fringe tree $T^v$ except $T^o=T$ (where $o$ is the root)
belongs to $\cht$. Hence, if $\gf(T)=1$, then $\Phi(T)=1$, and consequently,
\begin{align}\label{pat4}
  \Cov\bigpar{\gf(\Tgl),\Phi(\Tgl)}
=
\E \gf(\Tgl) - \E \gf(\Tgl) \E\Phi(\Tgl)
.\end{align}
Similarly, $\Var\gf(\Tgl)=\E\gf(\Tgl)-(\E\gf(\Tgl))^2$,
and thus \eqref{fV0} yields, using  \eqref{pat4},
\begin{align}\label{pat6}
  \fV(\gl) = \E \gf(\Tgl) - 2\E \gf(\Tgl) \E\Phi(\Tgl) + (\E\gf(\Tgl))^2.
\end{align}

Let $\cAx:=\bigcup_{n=0}^\infty\setoi^n$ be the set of finite strings of \setoi.
If $\ga=\ga_1\dotsm\ga_n\in\cAx$, define $P(\ga)$
as the probability that the random
string $\Xi\sss1$ begins with $\ga$, \ie, 
\begin{align}\label{pga}
  P(\ga):=\prod_{i=1}^{|\ga|} q^{1-\ga_i}p^{\ga_i},
\end{align}
where $|\ga|\ge0$ is the length of $\ga$.
We may regard the strings $\ga\in\cAx$ as the nodes in the infinite binary
tree, and, using again $\gf(\TI)=0$, 
it follows (see \cite[(2.25)]{SJ347}) that, 
using also 
\eqref{pat2},
\begin{align}\label{pat5}
  \E\Phi(\Tgl) &
= \sum_{\ga\in\cAx}\E\gf\bigpar{\Tqq{P(\ga)\gl}}
= \sum_{\ga\in\cAx}\fE\bigpar{P(\ga)\gl}
=\pi_t \sum_{\ga\in\cAx}\Po\bigpar{P(\ga)\gl;m}
.\end{align}
Let $\tsumaxx$ denote the sum over all $\ga\in\cAx$ with $|\ga|\ge1$,
\ie, over all strings $\ga$ except the empty string.
Then \eqref{pat6} and \eqref{pat5} yield, using \eqref{pat2} again and
\eqref{po}, 
\begin{align}\label{pat7}
\fV(\gl)&
=\pi_t\Po(\gl;m)
-2\pi_t^2\Po(\gl;m) \sum_{\ga\in\cAx}\Po\bigpar{P(\ga)\gl;m}
+\pi_t^2\Po(\gl;m)^2
\notag\\&
=\pi_t \frac{\gl^m}{m!}e^{-\gl}
-2\frac{\pi_t^2}{m!^2}\sum_{\ga\in\cAx}P(\ga)^m\gl^{2m}e^{-(1+P(\ga))\gl}
+ \pi_t^2 \frac{\gl^{2m}}{m!^2}e^{-2\gl}
\notag\\&
=\pi_t \frac{\gl^m}{m!}e^{-\gl}
-2\frac{\pi_t^2}{m!^2}\sumaxx P(\ga)^m\gl^{2m}e^{-(1+P(\ga))\gl}
- \pi_t^2 \frac{\gl^{2m}}{m!^2}e^{-2\gl}
.\end{align}
The Mellin transform is thus, for $\Re s>-m$,
by a simple calculation,
\begin{align}\label{pat8}
\MfV(s)&
=\frac{\pi_t}{m!}\gG(m+s)
-2\frac{\pi_t^2}{m!^2}\sumaxx P(\ga)^m(1+P(\ga))^{-2m-s}\gG(2m+s)
\notag\\&\qquad
- \frac{\pi_t^2 }{m!^2}2^{-2m-s}\gG(2m+s)
.\end{align}
By a binomial expansion we have, with absolutely convergent sums, \eg{}
by \eqref{paa} below,
\begin{align}\label{pat10}
  \sumaxx P(\ga)^m(1+P(\ga))^{-2m-s}&
= \sumaxx P(\ga)^m\sumko \frac{(-1)^k}{k!}\frac{\gG(2m+s+k)}{\gG(2m+s)}
P(\ga)^k
\notag\\&
=\sumko\frac{(-1)^k}{k!}\frac{\gG(2m+s+k)}{\gG(2m+s)}\sumaxx P(\ga)^{m+k}
.\end{align}
For any exponent $b$ and $\ell\ge1$ we have from the definition \eqref{pga},
letting $j$ be the number of 1s in $\ga$,
\begin{align}\label{paa}
  \sum_{\ga:|\ga|=\ell}P(\ga)^{b}
=\sum_{j=0}^\ell\binom\ell{j}\bigpar{p^jq^{\ell-j}}^b
=(p^b+q^b)^\ell.
\end{align}
Hence, summing over $\ell\ge1$, we obtain from \eqref{pat10}
\begin{align}\label{pat11}
\sumaxx P(\ga)^m(1+P(\ga))^{-2m-s}
=\sumko\frac{(-1)^k}{k!}\frac{\gG(2m+s+k)}{\gG(2m+s)} 
\frac{p^{m+k}+q^{m+k}}{1-(p^{m+k}+q^{m+k})}
.\end{align}
Finally, we obtain \eqref{MfV} from \eqref{pat8} and \eqref{pat11}.

In the aperiodic case, \eqref{psiX0} holds, and
thus \eqref{ttmeann}--\eqref{tt+hgs} yield
\eqref{ttmeann0}--\eqref{tt+hgs0}, using \eqref{MfE} and \eqref{MfC}.
\end{proof}

\begin{remark}\label{Rsqrt}
We do not claim that the error term $o(n)$ in \eqref{ttmeann} is $o(\sqrt{n})$
so that $\E N_t(\hTn)$ may be replaced by $H\qw\psiE(\log n)$ in
\eqref{ttn}.
It is known that in the corresponding results for the size of a random trie,
this is in general \emph{not} true, see \cite{FlajoletRV} and 
\cite[Appendix C]{SJ347}. 
We leave it as an open problem whether the same may happen here too.
\end{remark}

\begin{remark}\label{Rmulti}
\refT{TP} extends to multivariate limits for several full binary trees $t_i$
by the Cram{\'e}r--Wold device, \ie, by considering linear combinations of 
different $N_{t_i}(\hTn)$, \cf{} \cite[Remark 3.10]{SJ347}.
In general, 
also with trees $t_i$ of different sizes, 
asymptotic covariances can be found by
calculations similar to the ones above for the variance; we omit the details.
\end{remark}

It remains to calculate $\pi_t$. We have the following result, which is
closely related to \cite[Lemma 1]{Ischebeck}.
\begin{lemma}\label{LPT}
Let $t$ be a full binary tree with $\esize{t}=m\ge1$.
Then, 
\begin{align}\label{lpt}
\pi_t := \P(\TR_m\in\cht)= 
m!\, q^{\LPL(t)}p^{\RPL(t)}
\prod_{k=2}^{m-1}\bigpar{1-\xpar{q^{k}+p^{k}}}^{-\nu_k(t)},
\end{align}
where $\nu_k(t)$ is the number of nodes $v\in t$ such that the 
fringe tree $t^v$ has leaf size $\esize{t^v}=k$.
\end{lemma}

\begin{proof}[Sketch of proof]
\citet[Lemma 1]{Ischebeck} calculates the closely related 
probability
$\P(\hTR_m=t)$, i.e., the probability that $\TR_m$ compresses to $t$. 
(The only difference, apart from notation, is that the formula for 
$\P(\hTR_m=t)$ includes also $k=m$ in the product, which comes from
allowing $v$ to be the root in the argument below.)
The same argument applies here, so we will 
be brief. First, for a binary tree $t'$ with $m$ leaves 
such that no leaf has a parent with
outdegree 1,
we have $\TR_m=t'$ if and only if the $m$ random strings
$\Xi\sss1,\dots,\Xi\sss m$, taken in some order, have initial segments that
correspond to the $m$ leaves of $t'$ in the obvious way, 
and it follows easily that
\begin{align}\label{lt}
  \P(\TR_m=t')= m!\, q^{\LPL(t')}p^{\RPL(t')}. 
\end{align}
We have $\hTR_m\in\cht$ when $\TR_m$ is a tree $t'$ that can be obtained from   
$t$ by inserting paths of arbitrary lengths under each internal node 
$v\in t$ except the root. When summing the probabilities \eqref{lt} over all
$t'\in\cht$, these paths contribute for each $v$ with $\esize{t^v}=k\in[2,m-1]$
a factor 
\begin{align}
  \sum_{\ell=0}^\infty (p^k+q^k)^\ell 
= \frac{1}{1-(p^k+q^k)}
\end{align}
and the result follows.
\end{proof}

The asymptotic normality of $N_t(\hTn)$ yields corresponding results for the
distributions of fringe trees and extended fringe trees.
We note first a simple corollary.
\begin{corollary}\label{CA}
  Let $t$ be a full binary tree with $\esize{t}=m>1$.
Then
\begin{align}\label{ca1}
  N_t(\hTn)/n = \E N_t(\hTn)/n + \op(1) =H\qw\psiE(\log n) +\op(1),
\end{align}
with the periodic function $\psiE(t)$ as in \refT{TP}.
In particular, in the aperiodic case,
\begin{align}\label{ca2}
  N_t(\hTn)/n \pto  \frac{\pi_t}{m(m-1)H}.
\end{align}
\end{corollary}
\begin{proof}
  The first equality in \eqref{ca1} follows from  $\Var N_t(\hTn)=O(n)$
by Chebyshev's inequality,
and the second is \eqref{ttmeann}.

In the aperiodic case, we use \eqref{ttmeann0} instead of \eqref{ttmeann}
and obtain \eqref{ca2}.
\end{proof}

The size of the random trie $\Tn$ 
shows oscillations in the periodic case, see \eg{}
\cite{
KnuthIII, 
Mahmoud, 
FlajoletRV, 
FHZ2010, 
SJ242, 
FHZ2014, 
JS-Analytic, 
SJ347}.
However, the Patricia trie
$\hTn$ is a full binary tree with $n$ leaves, and thus has a fixed size 
$|\hTn|=2n-1$. (This is special to the binary case considered here.)
Hence we obtain from \eqref{tor2} and \refC{CA} immediately the following
for the random fringe tree.
We state results both conditioned on the tree $\hTn$ 
and unconditioned (i.e., results of {quenched} and {annealed}
type, repectively.)

\begin{corollary}\label{CB}
  Let $t$ be a full binary tree with $\esize{t}=m>1$.
Then
\begin{align}\label{cb1q}
\P(\hTnx=t\mid\hTn)  &= \frac1{2H}\psiE(\log n) +\op(1),
\\ \label{cb1a}
\P(\hTnx=t)  &= \frac1{2H}\psiE(\log n) +o(1),
\end{align}
with the periodic function $\psiE(t)$ as in \refT{TP}.
In particular, in the aperiodic case,
\begin{align}\label{cb2q}
\P(\hTnx=t\mid\hTn)  &\pto \frac{\pi_t}{2m(m-1)H},
\\\label{cb2a}
\P(\hTnx=t)  &\tend \frac{\pi_t}{2m(m-1)H}.
\end{align}

Furthermore, if\/ $|t|=1$, \ie, $t=\TI$, then, for any $p$,
\begin{align}\label{cb3}
\P(\hTnx=t\mid\hTn) & =
\P(\hTnx=t) =\frac{n}{2n-1} \to \frac12.
\end{align}
\end{corollary}
\begin{proof}
  The quenched versions \eqref{cb1q} and \eqref{cb2q} are the same as
  \eqref{ca1} and \eqref{ca2} by \eqref{tor2} conditioned on $\hTn$,
recalling $|\hTn|=2n-1$.
The annealed versions follow by taking expectations; note that the error
term $\op(1)$ in \eqref{cb1q} is bounded so dominated convergence applies
and shows that its expectation is $o(1)$.

Finally, \eqref{cb3} is trivial (but included for completeness).
\end{proof}

For the random extended fringe tree $\hTnxx$, we 
obtain similarly, for the probabilities $q(\hTn;t)$ defined in \refSS{SSext}.
\begin{corollary}\label{CC}
  Let $t$ be a full binary tree with $\esize{t}>1$.
Then
\begin{align}\label{cc1q}
\qsin(\hTn;t\mid\hTn)  &= \esize{t} H\qw\psiE(\log n) +\op(1),
\\ \label{cc1a}
 \qsin(\hTn;t)  &= \esize{t}H\qw\psiE(\log n) +o(1),
\end{align}
with the periodic function $\psiE(t)$ as in \refT{TP}.
In particular, in the aperiodic case,
\begin{align}\label{cc2q}
\qsin(\hTn;t\mid\hTn)  &\pto \frac{\pi_t}{(\esize{t}-1)H},
\\\label{cc2a}
 q(\hTn;t)  &\tend \frac{\pi_t}{(\esize{t}-1)H}.
\end{align}

For $\esize{t}=1$, $q(\hTn)=1$ by definition.
\end{corollary}

\begin{proof}
  Follows as \refC{CB} from \refC{CA}, now using
  \eqref{qsin2}--\eqref{qsin3} and recalling $\esize{\hTn}=n$.
\end{proof}

\begin{remark}\label{R+normal}
  We have here only stated first order results for the distributions of
  fringe trees and extended fringe trees. We similarly obtain from \refT{TP}
  also asymptotic normality of these distributions in the quenched version, 
meaning asymptotic normality of the conditional probabilities above.
\end{remark}

\begin{remark}\label{Rosc}
  \refCs{CB} and \ref{CC} show that in the periodic case there is
oscillation and no  limit distribution, although suitable subsequences
converge in distribution.
It is well-known that for some related functionals for tries, the
oscillations are numerically very small; this is true here too when $m$ is
small, but not for large $m$.
Consider the symmetric case $p=q=\frac12$; then
$\dpp=\log2$. 
(In other periodic cases, $\dpp$ is smaller and the oscillations are 
substantially smaller than in the symmetric case, but they still become
large for large $m$.)
In the Fourier series \eqref{psiX} for $\fE$, we have
by \eqref{MfE} the constant term $\MfE(-1)=\pi_t/m(m-1)$, and
if we normalize by this term,
for the term $k=1$ we have
\begin{align}\label{freja}
  \frac{\MfE(-1-\chio\ii)}{\MfE(-1)}
=\frac{\gG\xpar{m-1-\frac{2\pi}{\log 2}\ii}}{\gG(m-1)}.
\end{align}
For $m=2$, this ratio has absolute value
$|\gG(1-\frac{2\pi}{\log 2}\ii)|\doteq 4.9 \cdot 10^{-6}$,
and higher Fourier coefficients are much smaller.
Hence, the oscillations are in this case hardly of practical importance.
However, the absolute value of the ratio increases as $m$ increases:
for $m=3$ it is $\doteq 4.5 \cdot 10^{-5}$, for $m=4$ it is 
$\doteq2.1\cdot10^{-4}$ and for $m=100$ it is $\doteq0.66$;
in fact, the absolute value of the ratio
converges to 1 as \mtoo{} (see
\cite[5.11.12]{NIST}), and the same holds for every Fourier coefficient.
Hence we cannot always ignore the oscillations.

More precisely, still taking $p=q=\frac12$,
the normalized $(\log2)$-periodic
function $\psio(x):=\psiE(x)/\MfE(-1)$ is non-negative by
\eqref{ttmeann} and has by \eqref{psiX} and \eqref{pat3} Fourier coefficients
\begin{align}\label{freja2}
\hpsio(k)=  \frac{\MfE(-1-k\chio\ii)}{\MfE(-1)}
=\frac{\gG\xpar{m-1-\frac{2\pi k}{\log 2}\ii}}{\gG(m-1)}.
\end{align}
Let $\frax{x}:=x-\floor{x}$  
denote the fractional part of a real number $x$, and 
suppose that $m\to\infty$ along a subsequence such that 
$\frax{\lg m}=\frax{(\log m)/d}\to u\in\oi$.
(Recall that $d=\log2$.)
It then follows from \eqref{freja2} and 
\cite[5.11.12]{NIST} that, for any $k\in\bbZ$,
\begin{align}
\hpsio(k)\sim 
m^{-(2\pi k/\log 2)\ii}  
=e^{-2\pi k (\lg m)\ii}  
\to e^{-2\pi k u\ii}
=\widehat{\gd_{du}}(k),
\end{align}
where $\gd_{du}$ is a point mass at $du$.
Hence, by the continuity theorem,
the function $\psio(x)$ converges weakly 
(as a measure on the circle $\bbR/d\bbZ$) to $\gd_{du}$,
which roughly means that for large $m$, $\psio(x)$ is concentrated 
at $x$ close to $du\approx d\frax{\lg m}$.
Hence, still roughly, $\psi_E(\log n)$ in \eqref{ttmeann}
is large when $\frax{\lg n} =\frax{\log n/d}\approx \frax{\lg m}$, but small otherwise.
This should not be surprising. 
For a large $n$, in the first generations of  the construction of the trie
from $n$ strings in \refSS{SStries}, by the law of large numbers
almost exatly half of the strings are passed to the left child and half to
the right child.
Consequently, there will be many fringe trees in $\Tn$, and thus in $\hTn$,
of leaf size $\approx 2^{-j}n$ for integers $j$, but few for intermediate sizes,
until we get down to small sizes $m$. Thus, for a fixed large $m$, we expect
many fringe trees of leaf size $m$ when $n/m$ is close to a power of 2, which is
the same as $\frax{\lg n}\approx\frax{\lg m}$.
\end{remark}

\begin{remark}\label{Rcancel}
If we consider the ratio between the probabilities for two given trees $t$
of the same leaf
size, then there are no oscillations even in the periodic case,
since the oscillations  in \refCs{CB} and \ref{CC} for the two trees cancel
by \eqref{psiX} and \eqref{MfE}.
\end{remark}

\section{Extended binary search trees}\label{SEBST}

In this section we study fringe trees of the extended BST $\eBST_n$;
in the next section we study the more complicated case of the 
compressed BST $\hBST_n$.

As said in the introduction, \citet{Aldous-fringe} shows that the random
fringe trees $\BSTx_n$ converge in distribution to some limiting random
fringe tree $\BSTxoo$ as \ntoo, and
\citet{Devroye1991,Devroye2002} show asymptotic normality of the subtree
counts $N_t(\BST_n)$.
Furthermore,
it is shown in
\cite{Aldous-fringe} 
(and in \cite[Theorem 2]{Devroye2002}) that
the distribution of the limiting random fringe tree $\BSTxoo$ 
equals the mixture $\sumk\frac{2}{(k+1)(k+2)}\Law(\BST_k)$,
\ie,
for any set $\fA\subseteq\fT$,
\begin{align}\label{ald}
  \P\bigpar{\BSTxoo\in\fA} = \sumk \frac{2}{(k+1)(k+2)}\P\bigpar{\BST_k\in\fA}.
\end{align}

It is straightforward to use this to obtain the corresponding results for
the extended BST. 
For a full binary tree $t$ with $\esize{t}>1$, and thus $\isize{t}>0$,
let $t\internal$ denote the subtree ot $t$ consisting of the internal
nodes. Note that then $t=\overline{t\internal}$.

\begin{theorem}  \label{TEBST}
The extended BST has  a limiting random fringe tree $\eBSTxoo$
with distribution given by, 
for every full binary tree $t$ with $\isize{t}=k>0$,
\begin{align}\label{teb0}
  \P\bigpar{\eBSTxoo=t}
=
\tfrac12\P(\BSTxoo=t\internal)
=
\frac{1}{(k+1)(k+2)}\P(\BST_k=t\internal)
,\end{align}
and (for the trivial case $\isize{t}=0$)
\begin{align}\label{teb3}
  \P\bigpar{\eBSTxoo=\TI}
=\tfrac12.
\end{align}

Moreover, 
for $t=\TI$ we trivially have $N_\TI(\eBST_n)=n+1$,
and 
for any full binary tree $t$ with $\esize{t}>1$, 
$N_t(\eBST_n)$ is asymptotically normal:
there exist constants $\egb_t>0$ and $\egam_t>0$ such that, as \ntoo,
\begin{align}\label{tbst1}
  \frac{N_t(\eBST_n)-n\egb_t}{\sqrt n}\dto N(0,\egamm_t),
\end{align}
with convergence of mean and variance,
where
\begin{align}\label{cs2}
\egb_t
= 2\P\bigpar{\eBSTxoo=t}
= \P\bigpar{\BSTxoo=t\internal}
=\frac{2}{(k+1)(k+2)}\P(\BST_k=t\internal).
\end{align}
In particular, for every full binary tree $t$ (with $\egb_\TI:=1$),
\begin{align}\label{tbst2}
  N_t(\eBST_n)/n\pto \egb_t.
\end{align}
Consequently, the fringe tree distribution converges to the limit $\eBSTxoo$
both in quenched and annealed sense: For every full binary tree $t$,
\begin{align}\label{tbst3}
\P(\eBSTx_n=t\mid\eBST_n)&=\frac{N_t(\eBST_n)}{|\eBST_n|}
\pto \P(\eBSTxoo=t),
\\\label{tbst4}
\P(\eBSTx_n=t)&=\E\frac{N_t(\eBST_n)}{|\eBST_n|}
\to \P(\eBSTxoo=t)
.\end{align}
\end{theorem}

\begin{proof}
  First, $\eBST_n$ has $n$ internal nodes and $n+1$ leaves,
and thus $2n+1$ nodes; hence 
$N_\TI(\eBST_n)=n+1$ and
\begin{align}\label{eb1}
\P(\eBSTx_n=\TI)
=
\P(\eBSTx_n=\TI\mid\eBST_n)
=\frac{n+1}{2n+1}\to\frac12,
\end{align}
which yields 
\eqref{tbst3}--\eqref{tbst4} in this trivial case with
\eqref{teb3}.

Furthermore, if the fringe tree in $\BST_n$ of a node $v\in\BST_n$ 
is $\BST_n^v=t_1$, say,
then the fringe tree $\eBST_n^v$ in $\eBST_n$ of the same node is the
extended binary tree $\et_1$; i.e., $\eBST_n^v=\overline{\BST_n^v}$.
Consequently,
for any binary tree $t_1$,
$N_{\et_1}(\eBST_n)=N_{t_1}(\BST_n)$.
If we here replace $t_1$ by $t\internal$ for a full binary tree $t$ with
$\esize{t}>1$, and thus $\isize{t}>0$, we
obtain
\begin{align}\label{cs1}
  N_{t}(\eBST_n)
=  N_{\overline{t\internal}}(\eBST_n)
=N_{t\internal}(\BST_n)
.\end{align}
The asymptotic normality \eqref{tbst1} follows immediately from \eqref{cs1}
and the asymptotic normality of $N_{t\internal}(\BST_n)$ proved by 
\citet{Devroye1991,Devroye2002}; the fact that the 
variance is of order $\Theta(n)$ and thus $\egam_t>0$ is 
implicit in \cite[Theorem 5 and its proof]{Devroye2002},
and stated explicitly in \cite[Theorem 1.22]{SJ296}.

The convergence in probability \eqref{tbst2} follows immediately from
\eqref{tbst1}. Furthermore, \eqref{cs1} and \eqref{ald} yield, if
$\isize{t}=k>0$, 
\begin{align}\label{cs4}
 \frac{N_t(\eBST_n)}n=\frac{N_\tint(\BST_n)}n
\pto \P(\BSTxoo=\tint)
=\frac{2}{(k+1)(k+2)}\P(\BST_k=\tint),
\end{align}
and thus
\begin{align}\label{cs5}
 \frac{N_t(\eBST_n)}{|\eBST_n|}
=
 \frac{N_t(\eBST_n)}{2n+1}
\pto\tfrac12 \P(\BSTxoo=\tint)
=\frac{1}{(k+1)(k+2)}\P(\BST_k=\tint).
\end{align}
This yields \eqref{tbst3} with \eqref{teb0} when 
when $\esize{t}>1$;
the trivial case $\esize{t}=1$ was shown in \eqref{eb1} above.
We see also \eqref{cs2} by comparing \eqref{tbst2} and \eqref{cs4}.
Finally, the annealed version \eqref{tbst4} follows  (as always) by taking
expectations in the quenched version \eqref{tbst3}.
\end{proof}

\begin{remark}\label{RBSTmulti}
\refT{TEBST} extends to multivariate limits for several full binary trees $t_i$
by the same method, since multivariate limit theorems are known for
$\BSTx_n$, see \cite[Theorem 1.22]{SJ296}.
\end{remark}

\begin{remark}\label{RBSTvar}
  Explicit formulas for the asymptotic
variance $\egamm_t$ in \eqref{tbst1}, and for covariances in the
multivariate version, follow from \eqref{cs1} and \cite[(1.10)--(1.11)]{SJ296},
see also \cite[Theorem 1]{Devroye1991}.
\end{remark}

For the random extended fringe tree $\eBSTxx_n$
we obtain
for the probabilities $q(\eBST_n;t)$ defined in \refSS{SSext}:
\begin{corollary}\label{EBSTC}
  Let $t$ be a full binary tree with $\isize{t}=k>0$,
Then
\begin{align}\label{ebstc2q}
\qsin(\eBST_n;t\mid\eBST_n)  &\pto \qsin(\eBSTxxoo;t)
=2\esize{t}\P(\eBSTxoo=t)
=\frac{2\esize{t}}{(k+1)(k+2)}\P(\BST_k=t\internal)
,\\\label{ebstc2a}
 q(\eBST_n;t)  &\tend \qsin(\eBSTxxoo;t) 
.\end{align}
\end{corollary}
\begin{proof}
  Immediate from \refL{Lfringe}, \eqref{tbst3} and \eqref{teb0}.
\end{proof}

\section{Compressed binary search trees}\label{SBST}

In this section we study fringe trees of the compressed BST $\hBST_n$.
We use again the results by 
\citet{Aldous-fringe} and 
\citet{Devroye1991,Devroye2002} on fringe trees in the BST, and in
particular \eqref{ald}, combined with \eqref{od2} and 
arguments as in \refS{SPat}.

\begin{theorem}  \label{THBST}
Let $t$ be a full binary tree. 
Then $N_t(\hBST_n)$ is asymptotically normal:
there exist constants $\hgb_t>0$ and $\hgam_t>0$ such that, as \ntoo,
\begin{align}\label{thbst1}
  \frac{N_t(\hBST_n)-n\hgb_t}{\sqrt n}\dto N(0,\hgamm_t),
\end{align}
with convergence of mean and variance.
In particular,
\begin{align}\label{thbst2}
  N_t(\hBST_n)/n\pto \hgb_t.
\end{align}
Furthermore, 
there exists a limiting fringe tree distribution given by a
random full binary tree $\hBSTxoo$ such that for every $t\in\hfT$,
\begin{align}\label{thbst0}
 \P(\hBSTxoo=t)= \frac{3}{2}\hgb_t
,\end{align}
and
(quenched version)
\begin{align}\label{thbst3}
\P(\hBSTx_n=t\mid\hBST_n)&=\frac{N_t(\hBST_n)}{|\hBST_n|}
\pto \P(\hBSTxoo=t)
\intertext{and (annealed version)}\label{thbst4}
\P(\hBSTx_n=t)&=\E\frac{N_t(\hBST_n)}{|\hBST_n|}
\to \P(\hBSTxoo=t)
.\end{align}
\end{theorem}

  We conjecture that also all higher moments converge in \eqref{thbst1},
but we have not pursued this and leave it as an open problem.

\begin{proof}
  We use again \eqref{od2}; thus $N_t(\hBST_n)=\GF(\BST_n)$
where $\GF$ is defined by \eqref{od0}--\eqref{od1}.
The asymptotic normality \eqref{thbst1} (with convergence of mean and variance)
then follows from
\cite[Corollary 1.15]{SJ296}; see also \cite{Devroye2002} for similar results.
The convergence in probability \eqref{thbst2} is an immediate consequence.

Furthermore, \cite[Corollary 1.15 and (1.24)]{SJ296} yield
\begin{align}\label{ulla1}
  \hgb_t = \sumk \frac{2}{(k+1)(k+2)}\E \gf(\BST_k)
= \sumk \frac{2}{(k+1)(k+2)}\P (\BST_k\in\cht).
\end{align}
Alternatively,
\eqref{thbst2} and dominated convergence yield, together with \eqref{od2},
\begin{align}\label{ulla3}
  \hgb_t&
=\lim_\ntoo \frac{\E N_t(\hBST_n)}{n}
=\lim_\ntoo \frac{1}{n}\sum_{v\in\BST_n}\indic{\BST_n^v\in\cht}
=\lim_\ntoo \P\bigpar{\BSTx_n\in\cht}
\notag\\&
= \P\bigpar{\BSTxoo\in\cht},
\end{align}
which agrees with \eqref{ulla1} by \eqref{ald}.
The union $\bigcup_{t\in\hfT}\cht$ of the sets $\cht$ over all full binary
trees $t$ is the set $\fTset{0,2}$
of all binary trees where the root has degree 2 or 0.
Hence, \eqref{ulla3} implies
\begin{align}\label{ulla2}
 \sum_{t\in\hfT} \hgb_t 
= \sum_{t\in\hfT}\P \bigpar{\BSTxoo\in\cht}
= \P\bigpar{\BSTxoo\in\fTset{0,2}}
=\lim_\ntoo \P\bigpar{\BSTx_n\in\fTset{0,2}}.
\end{align}
If $v\in \BST_n$, then the fringe tree $\BST_n^v\in\fTset{0,2}$ 
$\iff$
the degree $d(v)\in\set{0,2}$ 
$\iff$ $v\in\hBST_n$.
Hence, \eqref{ulla2} yields, using \eqref{jw2} and dominated convergence,
\begin{align}\label{ulla4}
 \sum_{t\in\hfT} \hgb_t 
=\lim_\ntoo \E\frac{|\set{v\in \BST_n: v\in\hBST_n}|}{n}
=\lim_\ntoo\E \frac{|\hBST_n|}{n}
=\frac23.
\end{align}
Consequently, $\sum_{t\in\hfT} \frac32\hgb_t =1$, so \eqref{thbst0} defines a
probability distribution on full binary trees.

Combining \eqref{thbst2} and \eqref{jw2} we obtain
\begin{align}
  \frac{N_t(\hBST_n)}{|\hBST_n|}\pto \frac32\hgb_t.
\end{align}
The quenched and annealed convergence in distribution \eqref{thbst3} and
\eqref{thbst4} then follow from \eqref{tor3} and \eqref{tor2}
(and  dominated convergence again).

It remains to show that $\hgb_t>0$ and $\hgam_t>0$. For $\hgb_t$, this is
immediate from \eqref{ulla1}.
For $\hgam_t$, let $m:=\esize t$ and note first that we may assume $m\ge2$,
since for $t=\TI$, we have $N_{\TI}(\hBST_n)=\esize{\hBST_n}=\esize{\BST_n}$,
and thus $\hgamm_\TI=2/45$ by \eqref{jw3}.
Fix a suitable $k\ge m$; we may take $k=4$ if $m=2$ and $k=m$ if $m\ge3$.
Assume $n\ge 2k-1$. In the construction of the binary search tree $\BST_n$
in \refSS{SSBST}, stop at every node that receives exactly $2k-1$ items.
At each such node, peek into the future to see whether the fringe tree at that
node will be a full binary tree (with $k$ leaves), or it will contain some
node of outdegree 1; in the latter case, continue the recursive constraction
at this node too, but in the first case, just mark the node and leave it.
The result is a subtree of $\BST_n$ that we denote by $\BSTY_n$; it has a
number $N'_k$ of marked nodes with $2k-1$ items each, and we recover $\BST_n$
by replacing each marked node by a random full binary tree with $k$ leaves
(more precisely, a copy of $\BST_{2k-1}$ conditioned on being a full binary
tree); denote these trees by  $T_1,\dots,T_{N'_k}$.
Every fringe tree $\BST_n^v$ that belongs to $\cht$, and thus has $m$ leaves,
either lies completely in $\BSTY_n$ or in one of the $N'_k$  trees $T_i$;
furthermore,
any fringe tree of $T_i$ that belongs to $\cht$ has to be a copy of $t$.
Thus we have
\begin{align}\label{gab1}
  N_t(\hBST_n)=N'' +\sum_{i=1}^{N'_k} N_t(T_i)
\end{align}
where $N''$ is determined by $\BSTY_n$. 
Condition on $\BSTY_n$ (which also determines $N'_k$).
Then the trees $T_i$ are (conditionally) independent and identically
distributed,
and, by our choice of $k$, $0<\P(N_t(T_i)=1)<1$ so
$c:=\Var N_t(T_i)>0$.
Hence \eqref{gab1} implies that the conditional variance
\begin{align}\label{gab2}
\Var\bigsqpar{  N_t(\hBST_n)\mid\BSTY_n} 
=\sum_{i=1}^{N'_k}\Var N_t(T_i)    
= c N'_k
.\end{align}
The number $N'_k$ equals the number of fringe trees of $\BST_n$ that are
full binary trees of size $2k-1$; we let $\hfT_k$ be the set of all such full
binary trees.
Thus 
\begin{align}
N'_k/n =\P\bigpar{\BSTx_n\in\hfT_k\mid \BST_n}.  
\end{align}
Hence we obtain from the known convergence
$\BSTx_n\dto\BSTxoo$ 
\begin{align}\label{gab4}
\E N'_k/n =\E\P\bigpar{\BSTx_n\in\hfT_k\mid \BST_n}
=\P\bigpar{\BSTx_n\in\hfT_k}
\to
\P\bigpar{\BSTxoo\in\hfT_k}
=:c'
,\end{align}
where it follows from \eqref{ald} that $c'>0$.
Recall also the law of total variance: 
for any square-integrable random variable
$X$ and random variable (or $\gs$-field) $Y$ 
\begin{align}\label{gab5}
\Var X = \E\Var[X\mid Y] + \Var \E[X\mid Y]
\ge \E\Var[X\mid Y] .
\end{align}
Consequently, \eqref{gab2} and \eqref{gab4} yield
\begin{align}\label{gab3}
\Var\bigsqpar{N_t(\hBST_n)}
\ge
\E\Var\bigsqpar{  N_t(\hBST_n)\mid\BSTY_n} 
= c \E N'_k =cc' n + o(n)
.\end{align}
The convergence of variance in \eqref{thbst1} thus yields
$\hgamm_t \ge cc'>0$.
\end{proof}

\begin{remark}\label{RmultiBST}
\refT{THBST} extends to multivariate limits for several full binary trees
by the Cram\'er--Wold device; we omit the details.
\end{remark}

For the random extended fringe tree $\hBSTxx_n$
we obtain
for the probabilities $q(\hBST_n;t)$ defined in \refSS{SSext}:
\begin{corollary}\label{BSTC}
  Let $t$ be a full binary tree.
Then
\begin{align}\label{hbstc2q}
\qsin(\hBST_n;t\mid\hBST_n)  &\pto \qsin(\hBSTxxoo;t)=3\esize{t}\hgb_t
,\\\label{hbstc2a}
 q(\hBST_n;t)  &\tend \qsin(\hBSTxxoo;t) 
.\end{align}
\end{corollary}
\begin{proof}
  Immediate from \eqref{thbst0}--\eqref{thbst3} and \refL{Lfringe}.
\end{proof}

The numbers $\hgb_t$ are given by \eqref{ulla1}, but since $\cht$ is an
infinite set, this formula is of limited use for explicit calculations.
In the remainder of the section, we give one way to find $\hgb_t$ and thus
the limiting fringe distribution $\hBSTxoo$ more explicitly.

\begin{problem}
  It seems possible that similar but more complicated arguments might make it
possible to compute also the asymptotic
variance $\hgamm_t$, but we have not pursued
this, and we leave it as an open problem.
\end{problem}

\begin{remark}
The random BST can be constructed by a simple continuous-time branching
process, which yields a simple description of the 
limiting random fringe tree $\BSTxoo$ as this branching process stopped at
an exponentially distributed random time \cite{Aldous-fringe}
(see also \cite[Example 6.2]{SJ306}). In principle, this leads to a
description of the compressed BST and its limiting fringe tree $\hBSTxoo$, 
but this becomes
more complicated and we have not been able to use it, for example to compute
$\hgb_t$ and the probabilities in \eqref{thbst0}; we therefore do not 
give the details. 
The rather complicated explicit values of $\hgb_t$ for small $t$
calculated in \refS{Sex} (from \refT{TG} below)
also suggest that no really simple description of $\hBSTxoo$ exists.
\end{remark}

\subsection{Computing $\hgb_t$}\label{SSgb}

We define a generating function for binary trees and sets of binary trees as
follows. 
For a binary tree $T$, let 
\begin{align}\label{pT}
 \ppi_T&:=\P(\BST_{|T|}=T),
\\\label{FT}
  F_T(x)& := \ppi_Tx^{|T|}
.\end{align}
For any set $\fTo$ of binary trees, let
\begin{align}\label{FfT}
  F_{\fTo}(x):=\sum_{T\in\fTo} F_T(x).
\end{align}

These generating functions will help us to compute fringe tree probabilities
by the following simple formula for the BST.

\begin{lemma}
  \label{LF}
If\/ $\fTo$ is a set of binary trees, then
\begin{align}\label{lf}
  \P(\BSTxoo\in\fTo)
= 2\intoi F_\fTo(x)(1-x)\dd x.
\end{align}
\end{lemma}
\begin{proof}
  By \eqref{FfT} and linearity, it suffices to consider the case when
  $\fTo=\set{T}$ for a single binary tree $T$.
Let $k:=|T|$.
Then, by \eqref{ald},
\begin{align}
  \P(\BSTxoo=T)
& = \frac{2}{(k+1)(k+2)}\P(\BST_k=T)
=2\ppi_T\intoi x^k(1-x)\dd x
\notag\\&
= 2\intoi F_T(x)(1-x)\dd x,
\end{align}
which shows \eqref{lf}.
\end{proof}

Let $t$ be a full binary tree, and recall that  $\chtt$ is the set of all binary
trees that contract to $t$, and  $\cht$ the subset of all such
trees where the root has degree 2 or 0. 
In other words, $\cht$ is the set of all
binary trees that can be obtained from $t$ by replacing any edge by a
path, 
and $\chtt$ is the set of all binary trees 
that can be obtained from these 
by adding a path of length $\ell\ge0$ to the root.
Define the corresponding generating functions
\begin{align}\label{Gt}
  G_{t}(x)&:=F_{\chtt}(x),
\\\label{Ht}
  H_{t}(x)&:=F_{\cht}(x)
.\end{align}

We state a series of lemmas to help us compute these generating functions.

\begin{lemma}\label{CL1}
  If\/ $T$ is a path with $l\ge1$ nodes, then
  \begin{align}\label{cl1}
F_T(x)
=\frac{x^\ell}{\ell!}
.  \end{align}
\end{lemma}
\begin{proof}
  By the construction of the BST,
  \begin{align}\label{cl1a}
\ppi_T=
    \P(\BST_\ell=T) = \frac{1}{\ell}\cdot\frac{1}{\ell-1}\dotsm\frac{1}{1}
=\frac{1}{\ell!},
  \end{align}
since a given path $T$ is obtained by exactly one choice of pivot each time.
Hence, \eqref{cl1} follows by the definition \eqref{FT}.
\end{proof}

\begin{lemma}\label{CL2}
If\/ $\fTPath$ is the set consisting of all paths with any number $\ell\ge1$
of nodes, then 
  \begin{align}\label{cl2}
F_\fTPath(x)
= \tfrac12\bigpar{e^{2x}-1}
.  \end{align}
\end{lemma}
\begin{proof}
There are $2^{\ell-1}$ different paths with $\ell$ nodes.
Thus \refL{CL1} yields,
letting $P_\ell$ denote any path with $|P_\ell|=\ell$,
  \begin{align}\label{cl2a}
F_\fTPath(x) = \suml 2^{\ell-1} F_{P_\ell}(x)
=\frac12\suml\frac{(2x)^\ell}{\ell!}
= \frac12\bigpar{e^{2x}-1}
.  \end{align}
\end{proof}

\begin{lemma}  \label{CL3}
Let $T$ be a binary tree and let $T_1$ be a tree obtained by adding a path 
with $\ell\ge1$ nodes to the root of $T$. Then
\begin{align}\label{cl3}
F_{T_1}(x)=\intox F_T(y) \frac{(x-y)^{\ell-1}}{(\ell-1)!} \dd y 
.\end{align}
\end{lemma}
\begin{proof}
  Let $|T|=k$; then $|T_1|=k+\ell$.
In analogy with \eqref{cl1a}, the construction of the BST 
yields
  \begin{align}\label{cl3a}
\ppi_{T_1}&
=   \P(\BST_{k+\ell}=T_1) 
= \frac{1}{k+\ell}\cdot\frac{1}{k+\ell-1}\dotsm\frac{1}{k+1}\cdot\P(\BST_k=T)
\notag\\&
=\frac{k!}{(k+\ell)!}\ppi_T,
  \end{align}
since $T_1$ is obtained by exactly one choice of pivot each of the first
$\ell$ times, and then by the same choices as for $T$.
On the other hand, 
a standard Beta integral yields
\begin{align}\label{cl1b}
  \intox F_T(y)(x-y)^{\ell-1}\dd y&
=
  \ppi_T\intox y^k(x-y)^{\ell-1}\dd y
= \ppi_Tx^{k+\ell}\intoi z^k(1-z)^{\ell-1}\dd z
\notag\\&
= \ppi_T\frac{\gG(k+1)\gG(\ell)}{\gG(k+1+\ell)}x^{k+\ell}
= \ppi_{T}\frac{k!\,(\ell-1)!}{(k+\ell)!}x^{k+\ell}
.\end{align}
By \eqref{cl3a} and \eqref{FT}, this equals 
$(\ell-1)!\,x^{k+\ell}\ppi_{T_1}$,  
and thus \eqref{cl3} follows by the definition \eqref{FT}.
\end{proof}

\begin{lemma}  \label{CL4}
Let $\fTo$ be a set of binary trees and let $\fT_1$ be the set of 
trees obtained by adding a path with any number $\ell\ge1$ of nodes
to the root of any tree $T\in\fTo$. Then
\begin{align}\label{cl4}
F_{\fT_1}(x)
= 2\intox F_{\fTo}(y) e^{2(x-y)}\dd y
.\end{align}
\end{lemma}
\begin{proof}
  For each $\ell$, there are $2^\ell$ different paths of length $\ell$ that
  can be added 
(including the choice of edge from the end of the path to the former root).
Hence, by \refL{CL3}, summing over all $T\in\fTo$ and all possible paths,
\begin{align}\label{cl4a}
F_{\fTo}(x) = \sum_{T\in\fTo}\sum_{\ell\ge1}2^\ell
\intox F_T(y) \frac{(x-y)^{\ell-1}}{(\ell-1)!} \dd y 
= 2\intox F_{\fTo}(y) e^{2(x-y)}\dd y
.\end{align}
\end{proof}

\begin{lemma}  \label{CL5}
Let $T_\sL$ and $T_\sR$ be binary trees and let $T_1$ be a tree obtained by
taking a path with $\ell\ge1$ nodes and adding $T_\sL$ and $T_\sR$ as the
left and right subtrees of the last node in the path.
Then 
\begin{align}\label{cl5}
F_{T_1}(x)=\intox F_{T_\sL}(y)F_{T_\sR}(y) \frac{(x-y)^{\ell-1}}{(\ell-1)!} \dd y 
.\end{align}
\end{lemma}
\begin{proof}
This is similar to the proof of \refL{CL3}.
  Let $|T_\sL|=k_\sL$ and $|T_\sR|=k_\sR$; then $|T_1|=k_\sL+k_\sR+\ell$.
The same argument as for \eqref{cl3a} now yields
  \begin{align}\label{cl5a}
\ppi_{T_1}&
=   \P(\BST_{k_\sL+k_\sR+\ell}=T_1) 
= \frac{1}{k_\sL+k_\sR+\ell}
\dotsm\frac{1}{k_\sL+k_\sR+1}\cdot\P(\BST_{k_\sL}=T_\sL)\cdot\P(\BST_{k_\sR}=T_\sR)
\notag\\&
=\frac{(k_\sL+k_\sR)!}{(k_\sL+k_\sR+\ell)!}\ppi_{T_\sL}\ppi_{T_\sR},
  \end{align}
since if the first $\ell$ pivots have been chosen correctly to form
the desired path, with the last node having left and right subtrees of sizes
$k_\sL$ and $k_\sR$, respectively,
then the shapes of those subtrees are independent.
The rest of the proof is as for \refL{CL3} with only notational changes.
\end{proof}

\begin{lemma}  \label{CL6}
Let $\fT_\sL$ and $\fT_\sR$ be two sets of
binary trees 
and let $\fT_1$ be the set of 
trees obtained by taking
a path with any number $\ell\ge1$ of nodes
and adding two trees $T_\sL\in\fT_\sL$ and $T_\sR\in\fT_\sR$ as the
left and right subtrees of the last node in the path.
Then
\begin{align}\label{cl6}
F_{\fT_1}(x)
= \intox F_{\fT_\sL}(y)F_{\fT_\sR}(y) e^{2(x-y)}\dd y
.\end{align}
\end{lemma}
\begin{proof}
  By \refL{CL5} and summing over $T_\sL\in\fT_\sL$, $T_\sR\in\fT_\sR$,
and $\ell\ge1$,
  just as   in the proof of \refL{CL4}; note that for a given $\ell$, now
  there are  $2^{\ell-1}$ paths.
\end{proof}

Finally, we obtain our formula for $\hgb_t$, using the functions $G_t(x)$
for which we provide a recursion.

\begin{theorem}\label{TG}
  Let $t$ be a full binary tree.
  \begin{alphenumerate}
  \item \label{TG1}
The generating function $G_{t}(x)$ can be computed recursively as follows.
\begin{rommenumerate}
\item \label{TG1a}
If\/ $t=\TI$, then
  \begin{align}\label{tg1a}
    G_\TI (x) = \frac12(e^{2x}-1). 
  \end{align}
\item \label{TG1b}
If\/ $|t|>1$, and the root of $t$ has left and right
subtrees $t_\sL$ and $t_\sR$, then
\begin{align}
  \label{tg1b}
G_{t}(x) = \intox G_{t_\sL}(y)G_{t_\sR}(y)e^{2(x-y)}\dd y.
\end{align}
\end{rommenumerate}

\item \label{TG2}
Then $\hgb_t$ is given by:
  \begin{rommenumerate}
  \item \label{TG2a}
    If\/ $|t|=1$, \ie, $t=\TI$, then $\hgb_t=1/3$.
  \item \label{TG2b}
If\/ $|t|>1$, and the root of $t$ has left and right
subtrees $t_\sL$ and $t_\sR$, then
\begin{align}
  \label{tg2b}
\hgb_{t} = \intoi (1-x)^2G_{t_\sL}(x)G_{t_\sR}(x)\dd x.
\end{align}
\end{rommenumerate}
\end{alphenumerate}
\end{theorem}

\begin{proof}
\pfitemref{TG1}
The recursion \eqref{tg1a}--\eqref{tg1b} is
an immediate consequence of 
the definition \eqref{Gt} and
\refLs{CL2} and \ref{CL6}.

\pfitemref{TG2a}
If $t=\TI$, then $N_t(\hBST_n)=\esize{\hBST_n}$, and thus \eqref{jw1} and
  \eqref{thbst2} show $\hgb_t=1/3$.

\pfitemref{TG2b} 
\refL{CL5} with $\ell=1$ shows, by summing over all
$T_\sL\in\chtx{t_\sL}$  and $T_\sR\in\chtx{t_\sR}$,
and recalling \eqref{Ht},
\begin{align}\label{tg3}
  H_{t}(x)
=  F_{\cht}(x)
=\intox G_{t_\sL}(y)G_{t_\sR}(y) \dd y 
.\end{align}
Hence, \eqref{ulla3} and \refL{LF} yield
\begin{align}\label{tg4}
  \hgb_{t}&=\P(\BSTxoo\in\cht)
=2\intoi H_{t}(x)(1-x)\dd x
=2\iint_{0\le y\le x\le 1} (1-x) G_{t_\sL}(y)G_{t_\sR}(y) \dd y \dd x
\notag\\&
=\intoi (1-y)^2 G_{t_\sL}(y)G_{t_\sR}(y) \dd y,
\end{align}
which proves \eqref{tg2b}.
(Alternatively, one can use  integration by parts in \eqref{tg4}.)
\end{proof}

\begin{remark}
  Note that \eqref{tg1b} and \eqref{tg2b} hold for $t=\TI$ too if we define
  $G_\emptyset:=1$. 

Furthermore, \eqref{tg2b}, \eqref{tg1b}, and an integration by parts yield
the alternative formula 
\begin{align}
  \label{tg9}
\hgb_{t} &= \intoi (1-x)^2e^{2x}\cdot G_{t_\sL}(x)G_{t_\sR}(x)e^{-2x}\dd x
=- \intoi \Bigpar{(1-x)^2e^{2x}}'\cdot \Bigpar{e^{-2x}G_{t}(x)}\dd x
\notag\\&
= \intoi 2x(1-x)e^{2x}\cdot e^{-2x}G_{t}(x)\dd x
= \intoi 2x(1-x)G_{t}(x)\dd x.
\end{align}
This is perhaps more elegant than \eqref{tg2b}, but \eqref{tg2b} seems
better for calculations.
\end{remark}

\section{The critical beta-splitting tree}\label{SCB}

The (annealed) extended fringe tree distribution for the critical
beta-splitting random tree $\CB_n$ is described directly in 
\cite[Theorem 7]{Aldous-beta3},
see also  \cite[Sections 4.2 and 4.10]{Aldous-beta2}.
We use our notation in \refSS{SSext}, now for the trees $\CB_n$ and their
limiting fringe tree $\CBXXoo$;
recall that this is an infinite tree with a
unique infinite path $v_0,v_1,\dots$ from the root $v_0$ (really a leaf),
see \refR{Rsin}; 
the tree $\CBXXXoo i$ is the fringe tree of $\CBXXoo$
rooted at $v_i$ (with the natural
definition) and these fringe trees are nested.
Then
\cite[Theorem 7]{Aldous-beta3} says that
the leaf sizes $\esize{\CBXXXoo i}$, $i=0,1,\dots$, 
of these trees form a 
Markov chain with certain (explicit) transition probabilities; given these
leaf sizes, $n_i$ say,  the sibling of each tree $\CBXXXoo i$, i.e. the tree
$\CBXXXoo{i+1}\setminus(\CBXXXoo{i}\cup\set{v_{i+1}})$, 
has $n_{i+1}-n_i$ leaves and 
is a copy of $\CB_{n_{i+1}-n_i}$, with all these trees (conditionally)
independent. 

Alternatively, which is simpler for our purposes, we have by
\cite[(60) and (70)]{Aldous-beta3}, for any $m\ge2$ and 
with $h_{m-1}$ the harmonic number \eqref{harmonic},
\begin{align}\label{scb1}
  \frac{\E[N_m(\CB_n)]}{n}\to\frac{6}{\pi^2}\,\frac{h_{m-1}}{m(m-1)},
\end{align}
(This is a version of \cite[Theorem 3]{Aldous-beta3}, with other proofs
given in \cite{SJ386} and \cite{iksanovHD}.)

\begin{theorem}
  \label{TCB}
For any full binary tree $t$ with $m:=\esize{t}\ge2$,
\begin{align}\label{scb2}
  \frac{\E[N_t(\CB_n)]}{n}\to\frac{6}{\pi^2}\,\frac{h_{m-1}}{m(m-1)} \P(\CB_m=t)
.\end{align}
Hence,
\begin{align}\label{scb3}
\P(\CBx_n=t) \to
\P(\CBx_\infty=t) :=
  \begin{cases}
    \frac12, & t=\TI,
\\
\frac{3}{\pi^2}\,\frac{h_{m-1}}{m(m-1)} \P(\CB_m=t),
& \esize{t}\ge2
  \end{cases}
\end{align}
and, if\/ $m=\esize{t}\ge2$, 
\begin{align}\label{scb4}
\qsin(\CB_n;t)\to\qsin(\CBXXoo;t)
=\frac{6}{\pi^2}\,\frac{h_{m-1}}{m-1} \P(\CB_m=t)  
.\end{align}
\end{theorem}
\begin{proof}
First, \eqref{scb1} and
\eqref{tor11} yield \eqref{scb2}.

Next, \eqref{scb2} and \eqref{tor2} yield \eqref{scb3} for $\esize{t}\ge2$,
recalling that $|\CB_n|=2n-1$; these formulas also show the trivial case
$t=\TI$. 

Finally, \eqref{scb4} follows by the annealed version of \refL{Lfringe}
(which is proved in the same way),
or directly from \eqref{qsin7}.
\end{proof}

\begin{remark}
  By a summation by parts.
  \begin{align}
    \sum_{m=2}^\infty\frac{h_{m-1}}{m(m-1)}
=    \sum_{m=2}^\infty h_{m-1}\Bigpar{\frac{1}{m-1}-\frac{1}{m}}
=    \sum_{m=1}^\infty (h_m-h_{m-1})\frac{1}{m}
=\frac{\pi^2}{6},
  \end{align}
and thus the \rhs{} of \eqref{scb1} sums to 1, and is thus a probability
distribution on \set{m\in \bbN:m\ge2}.
Consequently, \eqref{scb3} really defines a probability distribution on the
set of all full binary trees $t$.
\end{remark}

We have no explicit formula for the probabilities $\P(\CB_m=t)$ in
\eqref{scb2}--\eqref{scb4},
but for small $m$ they are easily calculated directly from the definition.

The limit \eqref{scb3} is
a limit theorem for the annealed
distribution of the random fringe tree $\CBx_n$.
We show that the corresponding quenched result holds too.

\begin{theorem}
  \label{TCBq}
For any full binary tree $t$, 
\begin{align}\label{qscb3}
\P(\CBx_n=t\mid\CBx_n) \pto
\P(\CBx_\infty=t) 
\end{align}
given in \eqref{scb3}.
Hence, 
\begin{align}\label{qscb4}
\qsin(\CB_n;t\mid\CB_n)\pto\qsin(\CBXXoo;t)
\end{align}
given in \eqref{scb4}.
Furthermore, 
\begin{align}
  \label{qscb5}
\frac{N_t(\CB_n)}{n}\pto \qsin(\CBXXoo;t).
\end{align}
\end{theorem}

\begin{proof}
  Fix $t$. We may assume $\esize{t}\ge2$, since the results are trivial
  otherwise. Let
  \begin{align}\label{qa1}
    \mux{n}:=\E[ N_t(\CB_n)].
  \end{align}
\refT{TCB} shows that
\begin{align}\label{qa2}
\frac{\mux{n}}n
\to \mu
:=\qsin(\CBXXoo;t)
.\end{align}
Fix also a sequence of constants $b_n\to\infty$ such that
$b_n=o(n)$. Consider only $n$ such that $\esize{t}<b_n<n$.

We use the terminology of \cite{Aldous-beta2,Aldous-beta3}
and call the sets of leaves used in the recursive construction in
\refSSS{SSSbeta} \emph{clades}.
Thus, every node $v\in\CB_n$ corresponds to a clade, which equals the set
of leaves in the fringe tree $\CB_n^v$. 
The construction starts at the root with the clade $[n]$,
which is recursively split (randomly) into subclades.

We now construct the tree $\cD_n$ by recursive splits 
of the clade $[n]$  as
in \refSSS{SSSbeta}, but stop at each node where the corresponding clade
has size $<b_n$. Put a mark at these nodes, and then continue the
construction of $\cD_n$.
This gives a set $v_1,\dots,v_\nu$ of marked nodes (with $\nu$ random),
where each marked node $v_i$ corresponds to a clade $Z_i$;
$Z_i$ is just the set of leaves in the fringe tree $D_n^{v_i}$ rooted at $v_i$.
By construction, $\set{Z_1,\dots,Z_\nu}$ is a partition of $[n]$,
with $|Z_i|<b_n$ for every $i$.

Conditioned on $\nu$ and the sizes $\zeta_i:=|Z_i|$ of the clades $Z_i$,
the fringe trees $T_i:=\CB_n^{v_i}$ are independent with
$T_i\eqd \CB_{\zeta_i}$. Since we assume $b_n>\esize{t}$, every fringe tree
in $\CB_n$ that is a copy of $t$ is a fringe tree of one of the $T_i$.
Hence,
\begin{align}\label{qa3}
  N_t(\CB_n)=\suminu N_t(T_i).
\end{align}
It follows that, conditioned on $\nu$ and the clade sizes
$\zeta_1,\dots,\zeta_\nu$, 
\begin{align}\label{qa4}
\E[ N_t(\CB_n)\mid \nu,\zeta_1,\dots,\zeta_\nu]
= \suminu \E[N_t(T_i)\mid\zeta_i] = \suminu \mux{\zeta_i}.
\end{align}
Similarly, for the conditional variance,
using the trivial estimate $\Var[N_t(\CB_m)]\le m^2$,
\begin{align}\label{qa5}
&\E\Bigsqpar{\Bigpar{N_t(\CB_n)-\suminu\mux{\zeta_i}}^2\mid
  \nu,\zeta_1,\dots,\zeta_\nu}
=\Var[ N_t(\CB_n)\mid \nu,\zeta_1,\dots,\zeta_\nu]
\notag\\&\qquad
= \suminu \Var[N_t(\CB_{\zeta_i})\mid\zeta_i] 
\le \suminu\zeta_i^2
\le \suminu b_n\zeta_i
=b_n n
=o(n^2).
\end{align}
Hence,
\begin{align}
  \label{qa6}
\frac{N_t(\CB_n)}{n}-\frac{1}{n}\suminu\mux{\zeta_i} 
\pto0.
\end{align}
Let $\eps>0$.
By \eqref{qa2} there exists $M=M_\eps$ such that if $n\ge M$, then
$|\mux{n}-n\mu|<n\eps$. Hence,
noting that $0\le N_t(\CB_m)\le m$ for $m\ge1$, 
and thus $\mu(m)\le m$ and $\mu\le1$,
\begin{align}\label{qa7}
  \lrabs{\suminu\mux{\zeta_i}-n\mu}
= \lrabs{\suminu\bigpar{\mux{\zeta_i}-\zeta_i\mu}}
\le
\sum_{\zeta_i\ge M}\zeta_i\eps
+\sum_{\zeta_i<M}\zeta_i
\le
n\eps
+\sum_{\zeta_i<M}M.
\end{align}
With the notation in \refL{LCBq} below, we have
$\sum_{\zeta_i<M}M = M\sum_{j=1}^{M-1}X_j$
and thus 
\refL{LCBq} implies that with high probability 
(i.e., with probability $1-o(1)$),
\begin{align}\label{qa8}
 \sum_{\zeta_i<M}M
\le \eps n
.\end{align}
It follows from \eqref{qa7} and \eqref{qa8}, since $\eps>0$ is arbitrary, that
\begin{align}\label{qa9}
\frac{1}{n}\lrabs{\suminu\mux{\zeta_i}-n\mu}\pto0,
\end{align}
which together with \eqref{qa6} yields
\begin{align}\label{qa10}
  \frac{N_t(\CB_n)}{n}-\mu\pto0.
\end{align}
This is, by \eqref{qa2},  the same as \eqref{qscb5}.

Next, \eqref{qscb5}, \eqref{tor2}, and \eqref{scb3} yield \eqref{qscb3}.
Finally, \eqref{qscb4} follows by \refL{Lfringe}
or directly from \eqref{qsin7}, using \eqref{scb3}--\eqref{scb4}.
\end{proof}

\begin{lemma}
  \label{LCBq}
Let, as in the proof of \refT{TCBq}, $\zeta_1,\dots,\zeta_\nu$ be the sizes
of the clades obtained by stopping the construction of $\CB_n$ at 
clades of size less than $b_n$, for a given sequence $b_n\to\infty$.
For $j\ge1$, let
\begin{align}\label{lcbq1}
  Y_j:=\bigabs{\set{i:\zeta_i=j}},
\end{align}
the number of these clades that have size $j$.
Then, as \ntoo, for each fixed $j$,
\begin{align}\label{lcbq2}
  Y_j/n \pto0.
\end{align}
\end{lemma}

\begin{proof}
We continue to use the notation of the proof of \refT{TCBq}.
We assume, as we may, $j< b_n < n$.

Let $L^*$ be a uniformly random leaf in $\cD_n$.
Let $\pi_j$ be the probability that $L^*$ belongs to 
one of the marked clades of size $j$.
The total number of leaves in these clades is $jY_j$, and thus
\begin{align}
  \label{qb1}
\pi_j=\E \frac{jY_j}{n}=j\frac{\E Y_j}{n}.
\end{align}

Let $X_0=n, X_1, X_2,\dots$ be the decreasing sequence of sizes of the
clades that contain $L^*$. This is a Markov chain, stopped 
when it reaches 1; it is called the \emph{harmonic descent (HD)} chain 
\cite{Aldous-beta2, Aldous-beta3,SJ386,iksanovHD}, 
and has the transition
probabilities, as a simple consequence of \eqref{beta-1},
\begin{align}\label{qb2}
  p(i,j)=\frac{1}{h_{i-1}}\cdot\frac{1}{i-j},
\qquad i>j\ge1.
\end{align}
Let
 \begin{equation} \label{ani}
a(n,i) :=  
\P\bigpar{\text{the chain started at state $n$ is ever in state $i$}}.
 \end{equation}
Since the chosen leaf $L^*$ belongs to a marked clade if and only if the HD
chain makes a jump from some $X_i\ge b_n$ to $j$ (and this can happen at
most once), we have
\begin{align}\label{qb4}
  \pi_j=\sum_{k=b_n}^n a(n,k) p(k,j)
=\sum_{k=b_n}^n a(n,k) \frac{1}{h_{k-1}}\cdot\frac{1}{k-j}.
\end{align}
A lot is known about the occupancy measure $a(n,k)$, see again
\cite{Aldous-beta2, Aldous-beta3,SJ386,iksanovHD}, 
but here we only need the trivial observation that the \rhs{} of \eqref{qb4}
is increasing in $j<b_n$.
Hence, 
\begin{align}\label{qb5}
  \pi_j\le \pi_\ell,
\qquad 1\le j\le\ell<b_n.
\end{align}
Since every leaf belongs to exactly one marked clade, we have 
$\sum_{j=1}^{b_n-1}\pi_j=1$, and thus \eqref{qb5} implies
\begin{align}
  \pi_j \le \frac{1}{b_n-j},
\qquad 1\le j<b_n.
\end{align}
Consequently, for every fixed $j$, $\pi_j\to0$ as \ntoo, 
and thus by \eqref{qb1}
\begin{align}
  \E [Y_j/n]\to0,
\end{align}
which implies \eqref{lcbq2}.
\end{proof}

As in \cite{Aldous-beta2,Aldous-beta3}, we conjecture that a central limit
theorem for $N_t(\CB_n)$ holds, as for the other random trees studied here,
but we leave this as an open problem.
(The decomposition \eqref{qa3} in the proof of \refT{TCBq} might be helpful.)

\section{The uniform full binary tree}\label{Suni}
For comparison, we give also the corresponding values for
the uniform random full binary tree $\cU_n$ with $n$ leaves.
Note that this random tree is quite different from the other trees studied
here; 
for example, as is well-known, typically $\cU_n$ has height of order
$\sqrt n$, while $\hTn$, $\eBST_n$, and $\hBST_n$ have heights of order
$\log n$, and $\cD_n$ have height of order $\log^2n$ 
\cite{Aldous-beta1}, \cite[Section~3.13]{Aldous-beta2}.

The random full binary tree $\cU_n$ can be regarded as a conditioned 
Galton--Watson tree
with critical offspring distribution $\P(\xi=0)=\P(\xi=2)=\frac12$,
see \eg{} \cite{AldousII},
and thus it follows by \citet[Lemma 9]{Aldous-fringe} 
(see also \cite{SJ285})
that the asymptotic
fringe tree distribution is the corresponding (unconditioned) 
Galton--Watson tree, which in
this case simply means that for every full binary tree $t$,
\begin{align}\label{cux}
  \P(\cUx_n=t) \to 2^{-|t|} = 2^{1-2\esize{t}}
.\end{align}
We have also the quenched version 
and asymptotic normality of the fringe tree counts:
\begin{theorem}
  \label{TU}
Let $t$ be a full binary tree with $\esize{t}=m\ge1$.
Then, as \ntoo,
\begin{align}\label{tu1}
 \P(\cUx_n=t\mid \cU_n)=\frac{N_t(\cU_n)}{|\cU_n|} \pto 2^{-|t|} = 2^{1-2m}.
\end{align}
and consequently
\begin{align}\label{cuxx}
  \qsin(\cU_n;t\mid\cU_n) 
\pto \esize{t} 2^{1-|t|}
=m 2^{2-2m}
\end{align}

Furthermore, if\/ $\esize{t}=m>1$, then
$N_t(\cU_n)$ is asymptotically normal as \ntoo:
\begin{align}
\frac{N_t(\cU_n)-n2^{-|t|}}{\sqrt{n}}
  &\dto N\xpar{0,\gammu},\label{tu2}
\end{align}
with convergence of mean and variance,
where
\begin{align}\label{tu3}
  \gammu = 2^{1-2m}-(2m-1)2^{3-4m}>0.
\end{align}
\end{theorem}
\begin{proof}
The case $m=1$ is trivial, with $N_{\TI}(\cU_n)=n$ and $|\cU_n|=2n-1$.

For $m>1$, 
we have \eqref{tu1} by \eqref{tor3} and \cite[Theorem 7.12]{SJ264},
and then \eqref{cuxx} follows by \refL{Lfringe}.
The asymptotic normality \eqref{tu2}--\eqref{tu3}
follows by \cite[Corollary 1.8]{SJ285}
\end{proof}

\section{Examples for small fringe trees}\label{Sex}

We give here some examples of exact and numerical values for the limits
of $\P(T_n^*=t)$ and $q(T_n;t)$ that describe the 
asymptotic distribution of random fringe trees and extended fringe trees
when $T_n$ is 
one of the five random trees
$\hTn$, $\eBST_n$,  $\hBST_n$, $\cD_n$, $\cU_n$.
(These limits are related by \eqref{qsin12}, or its annealed version, but we
prefer to give explicit results for both.) 
We consider only the smallest fringe trees $t$, with $\esize{t}\le4$.
Since we consider only full binary trees $T_n$, we assume that $t$ is a full
binary tree; moreover, the case $\esize{t}=1$ is trivial, since then
for any sequence $T_n$ of full binary trees
\begin{align}\label{AA}
\P(\Tx_n=\TI\mid T_n)=\frac{\esize{T_n}}{|T_n|}
=\frac{\esize{T_n}}{2\esize{T_n}-1}
\to\frac12
\end{align}
provided $\esize{T_n}\to\infty$,
and similarly $q(T;\TI)=1$ for every $T$ by the definition \eqref{qsin1}.
Hence we consider the small trees in
\refF{F:trees} 
with the notations given there, and their mirror images which we may ignore
since they give the same result, with $p$ and $q$ exchanged for $\hTn$.
For simplicity, we state only the annealed versions of the results; note that
the (stronger) quenched results too hold by the results above.
Also for simplicity, for Patricia tries, we do not show the oscillations in
that appear in the periodic case (in particular in the symmetric case
$p=q=\frac12$); we give only the constant terms coming from the constant
term $\MfE(-1)$ in \eqref{psiX} and ignore the oscillating terms there
(which is the same as averaging $\psiE$ over a period). We denote this by
$\approx$ below; note that in the aperiodic case, $\approx$ thus means $\to$.
We use freely notation from the previous sections, and omit the simple
calculations.

\begin{remark}\label{Rclade2}
  If we ignore orientations and regard the random trees as cladograms,
see \refS{SS1def}, and for Patricia tries consider only the symmetric case
$p=\frac12$, then we do not have to consider mirror images, and also 
$t_{4b}$ disappears; instead
we should multiply the results below for $t_3$ by 2, and  for $t_{4a}$ by 4
(the numbers of possible orientations).
In particular, this should be noted when comparing the results below with
\cite[Figure~7]{Aldous-beta3} where this cladogram version of our
$q(\CBXXoo;t)$ is 
given for small $t$ (with $\esize{t}\le6$).
\end{remark}

\begin{figure}[t]
\setlength{\unitlength}{0.12in}
\begin{picture}(40,6)(0,-2)
\multiput(0,0)(2,0){2}{\circle*{0.5}}
\put(1,1){\circle*{0.5}}
\drawline(0,0)(1,1)(2,0)
\put(1,-2){$t_2$}

\multiput(6,0)(2,0){3}{\circle*{0.5}}
\put(8,2){\circle*{0.5}}
\put(9,1){\circle*{0.5}}
\drawline(6,0)(8,2)(9,1)(10,0)
\drawline(8,0)(9,1)
\put(8,-2){$t_3$}

\multiput(14,0)(2,0){4}{\circle*{0.5}}
\put(17,3){\circle*{0.5}}
\put(18,2){\circle*{0.5}}
\put(19,1){\circle*{0.5}}
\drawline(14,0)(17,3)(18,2)(19,1)(20,0)
\drawline(16,0)(18,2)
\drawline(18,0)(19,1)
\put(17,-2){$t_{4a}$}

\multiput(24,0)(2,0){4}{\circle*{0.5}}
\put(27,3){\circle*{0.5}}
\put(28,2){\circle*{0.5}}
\put(27,1){\circle*{0.5}}
\drawline(24,0)(27,3)(28,2)(30,0)
\drawline(26,0)(27,1)(28,2)
\drawline(28,0)(27,1)
\put(27,-2){$t_{4b}$}

\multiput(34,0)(2,0){4}{\circle*{0.5}}
\put(37,3){\circle*{0.5}}
\put(35,1){\circle*{0.5}}
\put(39,1){\circle*{0.5}}
\drawline(34,0)(35,1)(37,3)(39,1)(40,0)
\drawline(36,0)(35,1)
\drawline(38,0)(39,1)
\put(37,-2){$t_{4c}$}

 \end{picture}
\caption{Some small full binary trees.}
\label{F:trees}
\end{figure}
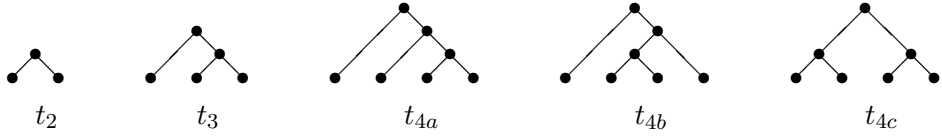

\subsection{The examples}\label{SSex}
We consider the trees $t_2,\dots,t_{4c}$ in \refF{F:trees} one by one;
for each of them we consider the five 
different random fringe trees studied above.

\begin{examplet}{t_2}\label{Et2}
For the Patricia trie $\hTn$, note first that
we have $\esize{t_2}=2$, 
$\LPL(t_2)=\RPL(t_2)=1$,  
$\nu_1(t_2)=2$, $\nu_2(t_2)=1$, and  $\nu_k(t_2)=0$, $k>2$.
Hence, 
\refL{LPT} yields
\begin{align}
  \pi_{t_2}=2pq.
\end{align}
(Which perhaps is more easily seen directly.)
\refCs{CB} and \ref{CC} then yield
\begin{align}
  \P(\hTnx=t_2) &\approx \frac{pq}{2H},
\\
\qsin(\hTn;t_2)&\approx \frac{2pq}{H}.
\end{align}
In particular, in the symmetric case $p=\frac12$, when $H=\log 2$,
\begin{align}
  \P(\hTnx=t_2) &\approx \frac{1}{8\log2}\doteq 0.1803, 
\\
\qsin(\hTn;t_2)&\approx \frac{1}{2\log2}\doteq 0.7213.  
\end{align}

For the extended BST $\eBST_n$, 
$\isize{t_2}=1$ and
\refT{TEBST} yields
\begin{align}\label{ebst2}
  \P(\eBSTx_n=t_2) &\to \frac{1}6
\doteq 0.1667,
\\\label{ebstq2}  
\qsin(\eBST_n;t_2)&\to \frac{2}3
\doteq0.6667
.\end{align}

For the compressed BST $\hBST_n$, 
\refT{TG} yields $G_{t_2}(x)=\frac{1}{8}e^{4x}-\frac12xe^{2x}-\frac18$ and
\begin{align}\label{gb2}
  \hgb_{t_2} = \tfrac1{128}e^4-\tfrac18 e^2 +\tfrac{233}{384}
\doteq 0.1097, 
\end{align}
Hence, \refT{THBST} and \refL{Lfringe} yield
\begin{align}\label{bst2}
  \P(\hBSTx_n=t_2) &\to 
\tfrac32\hgb_{t_2} 
=\tfrac3{256}e^4-\tfrac3{16} e^2 +\tfrac{233}{256}
\doteq 0.1645.   
\\\label{bstq2}  
\qsin(\hBST_n;t_2)&\to 
6\hgb_{t_2} 
=\tfrac3{64}e^4-\tfrac3{4} e^2 +\tfrac{233}{64}
\doteq 0.6581 
.\end{align}

For the critical beta-splitting random tree $\CB_n$, 
\refT{TCB} yields
\begin{align}\label{cbx2}
  \P(\CBx_n=t_2) &\to \frac{3}{2\pi^2} \doteq 0.1520,
\\\label{cbxx2}  
\qsin(\CB_n;t_2)&\to \frac{6}{\pi^2} \doteq 0.6079
.\end{align}

For the uniformly random full binary tree $\cU_n$, 
\eqref{cux} and \eqref{cuxx} yield
\begin{align}\label{cux2}
  \P(\cUx_n=t_2) &\to \frac{1}{8} = 0.125,
\\\label{cuxx2}  
\qsin(\cU_n;t_2)&\to \frac12=0.5.
\end{align}
\end{examplet}

\begin{examplet}{t_3}\label{Et3}
For $\hTn$,
we have $\esize{t_3}=3$, 
$\LPL(t_3)=2$, $\RPL(t_3)=3$, and $\nu_1(t_3)=3$,
$\nu_2(t_3)=1$, $\nu_3(t_3)=1$.
Hence, \refL{LPT} yields
\begin{align}
  \pi_{t_3}=\frac{6p^3q^2}{1-p^2-q^2}=\frac{6p^3q^2}{2pq}
=3p^2q.
\end{align}
\refCs{CB} and \ref{CC} then yield
\begin{align}
  \P(\hTnx=t_3) &\approx \frac{p^2q}{4H},
\\
\qsin(\hTn;t_3)&\approx \frac{3p^2q}{2H}.
\end{align}
In particular, in the symmetric case $p=\frac12$, when $H=\log 2$,
\begin{align}
  \P(\hTnx=t_3) &\approx \frac{1}{32\log2}\doteq 0.0451, 
\\
\qsin(\hTn;t_3)&\approx \frac{3}{16\log2}\doteq 0.2705.  
\end{align}

For $\eBST_n$, 
$\isize{t_3}=2$ and
\refT{TEBST} yields
\begin{align}\label{ebst3}
  \P(\eBSTx_n=t_3) &\to \frac{1}{24}
\doteq 0.0417,
\\\label{ebstq3}  
\qsin(\eBST_n;t_3)&\to \frac14=0.25
.\end{align}

For $\hBST_n$, \refT{TG} yields
\begin{align}\label{gb3}
\hgb_{t_3}=
\frac1{1728}e^6
-\frac 1{256}e^4
-\frac {3}{64}e^2
+\frac {2447}{6912}
\doteq 0.0279  
.\end{align}
Hence, \refT{THBST} and \refL{Lfringe} yield
\begin{align}\label{bst3}
  \P(\hBSTx_n=t_3) &\to 
\frac32\hgb_{t_3}
=
\frac1{1152}e^6
-\frac 3{512}e^4
-\frac {9}{128}e^2
+\frac {2447}{4608}
\doteq 0.0418   
,\\\label{bstq3}  
\qsin(\hBST_n;t_3)&\to
9\hgb_{t_3}
=
\frac1{192}e^6
-\frac 9{256}e^4
-\frac {27}{64}e^2
+\frac {2447}{768}
\doteq  0.2507  
.\end{align}

For 
$\CB_n$, \refT{TCB} yields,
since $\P(\CB_3=t_3)=\frac12$ (by symmetry),
\begin{align}\label{cbx3}
  \P(\CBx_n=t_3) &\to \frac{3}{8\pi^2} \doteq 0.0380,
\\\label{cbxx3}  
\qsin(\CB_n;t_3)&\to \frac{9}{4\pi^2} \doteq 0.2280
.\end{align}

For $\cU_n$, \eqref{cux} and \eqref{cuxx} yield
\begin{align}\label{cux3}
  \P(\cUx_n=t_3) &\to \frac{1}{32} = 0.03125,
\\\label{cuxx3}  
\qsin(\cU_n;t_3)&\to \frac3{16}=0.1875.
\end{align}
\end{examplet}

\begin{examplet}{t_{4a}}\label{Et4a}
For $\hTn$,
we have $\esize{t_{4a}}=4$, 
$\LPL(t_{4a})=3$, $\RPL(t_{4a})=6$, and
$\nu_1(t_{4a})=4$,
$\nu_2(t_{4a})=1$, $\nu_3(t_{4a})=1$, $\nu_4(t_{4a})=1$.
Hence, \refL{LPT} yields
\begin{align}
  \pi_{t_{4a}}=\frac{24p^6q^3}{(1-p^2-q^2)(1-p^3-q^3)}
=\frac{24p^6q^3}{2pq\cdot3pq}
=4p^4q.
\end{align}
\refCs{CB} and \ref{CC} then yield
\begin{align}
  \P(\hTnx=t_{4a}) &\approx \frac{p^4q}{6H},
\\
\qsin(\hTn;t_{4a})&\approx \frac{4p^4q}{3H}.
\end{align}
In particular, in the symmetric case $p=\frac12$, when $H=\log 2$,
\begin{align}
  \P(\hTnx=t_{4a}) &\approx \frac{1}{192\log2}\doteq 0.0075, 
\\
\qsin(\hTn;t_{4a})&\approx \frac{1}{24\log2}\doteq 0.0601. 
\end{align}

For $\eBST_n$,
$\isize{t_{4a}}=3$ and
\refT{TEBST} yields
\begin{align}\label{ebst4a}
  \P(\eBSTx_n=t_{4a}) &\to \frac{1}{120}
\doteq 0.0083,
\\\label{ebstq4a}  
\qsin(\eBST_n;t_{4a})&\to \frac{1}{15}
\doteq0.0667
.\end{align}

For $\hBST_n$, \refT{TG} yields
\begin{align}\label{gb4a}
\hgb_{t_{4a}}=
\frac1{32768}e^8
-\frac1{4608}e^6
-\frac {11}{512}e^2
+\frac {47503}{294912}
\doteq 0.0057 
.\end{align}
Hence, \refT{THBST} and \refL{Lfringe} yield
\begin{align}\label{bst4a}
  \P(\hBSTx_n=t_{4a})& \to
\frac32\hgb_{t_{4a}}=
\frac3{65536}e^8
-\frac1{3072}e^6
-\frac {33}{1024}e^2
+\frac {47503}{196608}
\doteq   0.0086  
,\\\label{hbstq4a}  
\qsin(\hBST_n;t_{4a})&\to 
12\hgb_{t_{4a}}=
\frac3{8192}e^8
-\frac1{384}e^6
-\frac {33}{128}e^2
+\frac {47503}{24576}
\doteq 0.0690  
.\end{align}

For 
$\CB_n$, 
\refT{TCB} yields,
since $\P(\CB_4=t_{4a})=\frac2{11}$ (by direct calculation),
\begin{align}\label{cbx4a}
  \P(\CBx_n=t_{4a}) &\to \frac{1}{12\pi^2} \doteq 0.0084,
\\\label{cbxx4a}  
\qsin(\CB_n;t_{4a})&\to \frac{2}{3\pi^2} \doteq 0.0675
.\end{align}

For $\cU_n$, \eqref{cux} and \eqref{cuxx} yield
\begin{align}\label{cux4a}
  \P(\cUx_n=t_{4a})& \to \frac{1}{128}\doteq 0.0078, 
\\\label{cuxx4a}  
\qsin(\cU_n;t_{4a})&\to \frac1{16} = 0.0625.
\end{align}
\end{examplet}

\begin{examplet}{t_{4b}}\label{Et4b}
For $\hTn$,
we have $\esize{t_{4b}}=4$,
$\LPL(t_{4b})=4$, $\RPL(t_{4b})=5$, and
$\nu_1(t_{4b})=4$,
$\nu_2(t_{4b})=1$, $\nu_3(t_{4b})=1$, $\nu_4(t_{4b})=1$.
Hence, \refL{LPT} yields
\begin{align}
  \pi_{t_{4b}}=\frac{24p^5q^4}{(1-p^2-q^2)(1-p^3-q^3)}
=\frac{24p^5q^4}{2pq\cdot3pq}
=4p^3q^2.
\end{align}
\refCs{CB} and \ref{CC} then yield
\begin{align}
  \P(\hTnx=t_{4b}) &\approx \frac{p^3q^2}{6H},
\\
\qsin(\hTn;t_{4b})&\approx \frac{4p^3q^2}{3H}.
\end{align}
All other results are by symmetry the same as for $t_{4a}$.  
\end{examplet}

\begin{examplet}{t_{4c}}\label{Et4c}
For $\hTn$,
we have $\esize{t_{4c}}=4$,
$\LPL(t_{4c})=\RPL(t_{4c})=4$, and
$\nu_1(t_{4c})=4$, $\nu_2(t_{4c})=2$, $\nu_{3}(t_{4c})=0$, $\nu_4(t_{4c})=1$.
Hence, \refL{LPT} yields
\begin{align}
  \pi_{t_{4c}}=\frac{24p^4q^4}{(1-p^2-q^2)^2}=\frac{24p^4q^4}{(2pq)^2}
=6p^2q^2.
\end{align}
\refCs{CB} and \ref{CC} then yield
\begin{align}
  \P(\hTnx=t_{4c}) &\approx \frac{p^2q^2}{4H},
\\
\qsin(\hTn;t_{4c})&\approx \frac{2p^2q^2}{H}.
\end{align}
In particular, in the symmetric case $p=\frac12$, when $H=\log 2$,
\begin{align}
  \P(\hTnx=t_{4c}) &\approx \frac{1}{64\log2}\doteq  0.0225, 
\\
\qsin(\hTn;t_{4c})&\approx \frac{1}{8\log2}\doteq 0.1803. 
\end{align}

For $\eBST_n$, 
$\isize{t_{4c}}=3$ and
\refT{TEBST} yields
\begin{align}\label{ebst4c}
  \P(\eBSTx_n=t_{4c}) &\to \frac{1}{60}
\doteq 0.0167,
\\\label{ebstq4c}  
\qsin(\eBST_n;t_{4c})&\to \frac{2}{15}
\doteq0.1333
.\end{align}

For $\hBST_n$, \refT{TG} yields
\begin{align}\label{gb4c}
\hgb_{t_{4c}}=
\frac1{16384}e^8
-\frac{1}{1728}e^6
+\frac{1}{1024}e^4
-\frac{1}{64}e^2
+\frac {54973}{442368}
\doteq  0.0106
.\end{align}
Hence, \refT{THBST} and \refL{Lfringe} yield
\begin{align}\label{bst4c}
  \P(\hBSTx_n=t_{4c})& \to
\frac32\hgb_{t_{4c}}=
\frac3{32768}e^8
-\frac{1}{1152}e^6
+\frac{3}{2048}e^4
-\frac{3}{128}e^2
+\frac {54973}{294912}
\doteq  0.0159,  
\\\label{hbstq4c}  
\qsin(\hBST_n;t_{4c})&\to 
12\hgb_{t_{4c}}=
\frac3{4096}e^8
-\frac{1}{144}e^6
+\frac{3}{256}e^4
-\frac{3}{16}e^2
+\frac {54973}{36864}
\doteq 0.1273  
.\end{align}

For 
$\CB_n$, 
\refT{TCB} yields,
since $\P(\CB_4=t_{4c})=\frac3{11}$ (by direct calculation),
\begin{align}\label{cbx4c}
  \P(\CBx_n=t_{4c}) &\to \frac{1}{8\pi^2} \doteq 0.0127,
\\\label{cbxx4c}  
\qsin(\CB_n;t_{4c})&\to \frac{1}{\pi^2} \doteq 0.1013
.\end{align}

For $\cU_n$, \eqref{cux} and \eqref{cuxx} yield (just as for $t_{4a}$)
\begin{align}\label{cux4c}
  \P(\cUx_n=t_{4c}) \to \frac{1}{128}\doteq 0.0078, 
\\\label{cuxx4c}  
\qsin(\cU_n;t_{4c})\to \frac1{16} = 0.0625.
\end{align}
\end{examplet}

We summarize the numerical values above in Tables \ref{tab:P} and
\ref{tab:qsin}.
In particular, note the large differences in the relative importance of
$t_{4a}$ and $t_{4c}$ for the five random full binary trees considered here:
the asymptotic ratio between the probabilities for $t_{4c}$ and $t_{4a}$ are
3 for symmetric Patricia tries (the oscillations cancel, see \refR{Rcancel}),
2 for extended BST,
$1.846\dots$ for compressed BST, 
$3/2$ for critical beta-splitting trees,
and 1 for uniform full binary trees.

In Table \ref{tab:qsin}, we include also for illustration empirical data
computed by David Aldous 
for a small set of 10 real cladograms  with a total of 995 species, 
as reported in \cite[Figure~7 and Section~5.3]{Aldous-beta3}.
(The values given here are adjusted for symmetries, see \refR{Rclade2}.)
The empirical data in \cite{Aldous-beta3} are given for all $t$ with
$\esize{t}\le6$, 
and we encourage the interested reader to extend the computations in the
examples below to all such $t$ and make comparisons. 

\begin{table}[h]
  \centering
  \begin{tabular}{l|l|l|l|l}
 & $t_2$ & $t_3$ & $t_{4a}$ & $t_{4c}$ \\
\hline
Patricia trie $\hTn$ ($p=q=\frac12$)& 0.1803 & 0.0451 & 0.0075 & 0.0225
\\    
Extended BST $\eBST_n$ & 0.1667 & 0.0417 & 0.0083 & 0.0167
\\
Compressed BST $\hBST_n$ & 0.1645 & 0.0418 & 0.0086 & 0.0159
\\
Critical beta-splitting $\cD_n$ & 0.1520 & 0.0380 & 0.0084 & 0.0127
\\
Uniform full binary tree $\cU_n$ & 0.125 & 0.0312 & 0.0078 & 0.0078
\\\hline
  \end{tabular}
  \caption{Limits or approximations of $\P(\Tx_n=t)$ for five random full
    binary trees.}
  \label{tab:P}
\end{table}

\begin{table}[h]
  \centering
  \begin{tabular}{l|l|l|l|l}
 & $t_2$ & $t_3$ & $t_{4a}$ & $t_{4c}$ \\
\hline
Patricia trie $\hTn$ ($p=q=\frac12$)& 0.7213 & 0.2705 & 0.0601 & 0.1803
\\    
Extended BST $\eBST_n$ & 0.6667 & 0.25 & 0.0667 & 0.1333
\\
Compressed BST $\hBST_n$ & 0.6581 & 0.2507 & 0.0690 & 0.1273
\\
Critical beta-splitting $\cD_n$ & 0.6079 & 0.2280 & 0.0675 & 0.1013
\\
Uniform full binary tree $\cU_n$ & 0.5 & 0.1875 & 0.0625 & 0.0625
\\\hline
real cladograms (10 samples) & 0.573 &  0.245 & 0.071 & 0.120
\\\hline
  \end{tabular}
  \caption{Limits or approximations of $\qsin(T_n;t)$ for five random full
    binary trees, together with some empirical data from real cladograms
\cite{Aldous-beta3}.
}
  \label{tab:qsin}
\end{table}

\newcommand\AAP{\emph{Adv. Appl. Probab.} }
\newcommand\JAP{\emph{J. Appl. Probab.} }
\newcommand\JAMS{\emph{J. \AMS} }
\newcommand\MAMS{\emph{Memoirs \AMS} }
\newcommand\PAMS{\emph{Proc. \AMS} }
\newcommand\TAMS{\emph{Trans. \AMS} }
\newcommand\AnnMS{\emph{Ann. Math. Statist.} }
\newcommand\AnnPr{\emph{Ann. Probab.} }
\newcommand\CPC{\emph{Combin. Probab. Comput.} }
\newcommand\JMAA{\emph{J. Math. Anal. Appl.} }
\newcommand\RSA{\emph{Random Structures Algorithms} }
\newcommand\DMTCS{\jour{Discr. Math. Theor. Comput. Sci.} }

\newcommand\AMS{Amer. Math. Soc.}
\newcommand\Springer{Springer-Verlag}
\newcommand\Wiley{Wiley}

\newcommand\vol{\textbf}
\newcommand\jour{\emph}
\newcommand\book{\emph}
\newcommand\inbook{\emph}
\def\no#1#2,{\unskip#2, no. #1,} 
\newcommand\toappear{\unskip, to appear}

\newcommand\arxiv[1]{\texttt{arXiv}:#1}
\newcommand\arXiv{\arxiv}

\def\nobibitem#1\par{}

\end{document}